\documentclass{amsart}

\usepackage{InsallLuckhardt_packages}
\usepackage{InsallLuckhardt}

\usepackage{etoc}

\excludecomment{workingNotes}
\excludecomment{Wasserstein}
\excludecomment{Banach}
\excludecomment{exercise}
\includecomment{exerciseSolved}

\title[Norms on Categories and Analogs of the CBS Theorem]{Norms on Categories and Analogs of the Schröder-Bernstein Theorem}

\date{\today}

\author[M.~Insall]{Matt Insall}
\address[M.~Insall]{Missouri S\&T, USA.}
\email{insall@mst.edu}

\author[D.~Luckhardt]{Daniel Luckhardt}
\address[D.~Luckhardt]{Department of Computer Science, University of Sheffield, UK.}
\email{daniel.luckhardt@mathematik.uni-goettingen.de}

\begin{document}

\maketitle

\begin{abstract}
    We generalize the concept of a norm on a vector space to one of a norm on a category. This provides a unified perspective on many specific matters in many different areas of mathematics like set theory, functional analysis, measure theory, topology, and metric space theory. We will especially address the two last areas in which the monotone-light factorization and, respectively, the Gromov-Hausdorff distance will naturally appear. 

In our formalization a Schröder-Bernstein property becomes an axiom of a norm which constitutes interesting properties of the categories in question. The proposed concept provides a convenient framework for metrizations.    
\end{abstract}

\etocruledstyle[2]{\scshape{} Contents }

%\renewcommand{\etocpartname}{\partname}%
% <1.08b hat \partname als Standard, aber dies arbeitet nicht gut mit frenchb
% daher ist der Standard jetzt nun einfach Teil. Für die deutsche Sprache
% können wir \partname oder Teil nehmen.

\etocsetnexttocdepth{subsubsection}
\begingroup
\etocsetlevel{section}{0}
\etocsetlevel{subsection}{1}
%\renewcommand*{\etocfontone}{\normalfont \normalsize}
%\renewcommand*{\etoctoclineleaders}
%{\hbox{\normalfont\normalsize\hbox to 1ex {\hss.\hss}}}
%\sloppy

\tableofcontents
\endgroup

\pagebreak

\section{Introduction}\label{sec:intro}
Norms form a corner stone of any quantitative research and belong to the foundational notions of most areas of mathematics.
The subject of this paper is the generalization of the concept of a norm on a vector space or, more generally, on an abelian group to one on a category.

Our attempt is preceded by many approaches using enriched monoidal categories going back to \textcite{Lawvere73} and an industry arising out of this \cite{Grandis07,Neeman20}.
Independently, \textcite{GhezLimaRoberts85} suggested a notion of normed *-category.
Recently, interleaving distance has gained some favour as a playing ground \cite{Scoccola20,BubenikdeSilvaScott18}.
Another recent approach outside of enriched or monoidal category theory is by \textcite{Kubis17, Perrone23}.

Our motivation in the present work is threefold, aspiring a framework
\begin{enumerate}
    \item\label{lst:motivation_CSB}
        to find analogs of the Cantor-Schröder-Bernstein theorem (CSB theorem) from set theory in other categories,
    \item\label{lst:motivation_metrization}
        for systematic and convenient metrization of families of equivalence classes of spaces, like the Gromov-Hausdorff space, moduli spaces, and representation spaces,
    \item 
        to prove general theorems in the developed categorical framework that give insights and concrete useful applications in many different areas of mathematics,
    \item\label{lst:motivation_extensionMorphisms}
        to work with categories with large classes of morphisms.
\end{enumerate}
The guiding example of our approach is the following norm on the category of sets:
To each function $f\colon X \to Y$ assign the non-negative extended real number
\begin{equation}
    \Norm{f}{inj} \coloneqq \supBd_{x \in X} \log( \# \setBuilder{y}{f(x)= f(y)})
\vspace{-1mm}
\end{equation}
where $\# $ assigns to a set the numbers of its elements (a member of $ \{0, 1, \ldots, \infty\} $) and $
    \supBd[a]\limits_{x\in X} f(x) = \sup \{a\} \cup \setBuilder{f(x)}{x \in X} 
$. 
Note that $f$ is injective if and only if $\Norm{f}{inj} = 0$.
This is to say that $\Norm{}{inj}$ is measuring the deviation from being injective.
Hence the idea for the generalization of the CSB theorem is that in a normed category $(\cat, \norm{})$ two objects $X$ and $Y$ are isomorphic as soon as there are morphisms $f\colon X \to Y$ and $g\colon Y \to X$ such that $\norm{X} = \norm{Y} = 0$.

Returning to our list of motivations, we mention with regard to the first one that the CSB theorem is a fundamental theorem in set theory stating that there is a bijection between two sets as soon as there are injective maps between the sets both ways. In conceptual terms it states that if $X$ can be embedded into $Y$ and vice-versa, then $X$ and $Y$ are isomorphic. A direct formalization of this conceptual idea in category theoretic terms would be the property that in a category $\cat$ two objects $X,Y \in \cat_0$ are isomorphic as soon as there are monomorphisms $X\to Y$ and $Y\to X$. 
But, unfortunately, most categories considered in practice do not have this property, cf.\ \cite{Laackman10}. 
Though, there are some notable exceptions, which include measure spaces \cite[\S~3.3]{Srivastava98} and a noncommutative version thereof for von Neumann algebras \cite[Proposition~6.2.4]{KadisonRingrose97}.
Efforts to generalize the CBS theorem in alternative set-ups have recently revived, including results for categories of universal algebras \cite{Freytes19} and in homotopy type theory and boolean $\infty$-topoi \cite{Escardo20} including a formalization in Agda \cite{Escardo20agda}. Of course, one can also weaken the property by replacing ``monomorphism'' by some stronger or related notion of morphism. In our approach it will be a morphism with vanishing norm.

As for the second motivation, the problems with doing metrizations in practise are often that they become very technical, involve arbitrary choices, and basic properties like the triangle inequality or completeness become hard to check. The theory of uniform structure can be seen as an attempt to abstract from these choices, but it lacks a measure of the size of entourages. In many examples we present it will turn out that the norms can be defined in terms of a capacity, a quantity measuring the size of a subobject of an object $X \in \cat_0$. 
It is natural to proceed by a category theoretical approach by looking at many examples:
representatives of a point (i.e.\ equivalence classes of spaces) are objects of a category and morphisms are comparison maps.

As for the last motivation, note that when working in some area of mathematics using the language of category theory  one often has to limit the class of morphisms under consideration. For instance the category of metric spaces is normally defined with morphisms to be non-expansive map in order to guarantee nice properties like existence of limits.

\subsection{Organization of the paper}

As for notation the reader is invited to consult \cref{sec:notation}.
We start the paper by introducing the basic definitions of seminorm, norm, and constructions based thereon (cf.\ \cref{sec:def}).
A categorically minded reader may enjoy complementing this perspective by a 2-categorical point of view elaborated in \cref{sec:2cat_viewPoint}.
Acquaintance with these definitions is facilitated in \cref{sec:examples} by a plethora of examples.
Thereafter we turn to the notation of capacities in \cref{sec:capacities} that will be the convenient framework for most seminorms under investigation in this paper. On this occasion we will also provide a first overview of seminorm central to our investigation in \cref{ssec:crucialSeminorms}.

\Cref{sec:top} addresses the category of topological spaces.
We will introduce a seminorm $\Norm{}{top} = \Norm{}{comp} + \Norm{}{dim}$, where $ \Norm{}{comp} $ measures the increase in the number of connected components when pulling back a subset of the codomain to a subset of the domain and $ \Norm{}{dim} $ measures the increase in dimension under such a transition. In other terms these seminorms measure the deviation of a map from being monotone or light respectively. The classic monotone-light factorization theorem implies that, when restricting to compact metrizable spaces, $\Norm{}{top}$ is a norm.

The subject of \cref{sec:Met} is the category of compact metric spaces $\Cat{Met}$ endowed with all multi-valued set theoretic maps as morphisms.
The notion of dilatation gives rise to the seminorm 
\[
                \Norm{f}{diam}
    =           \supBd \setBuilder{\spMet(x, x') - \spMet*(y, y')}{x, x' \in M, y\in f[x], y'\in f[x']}
\text.
\vspace{-1mm}
\]
Generalizing a classical theorem of Feudenthal and Hurewicz we show that $\Norm{f}{diam}$ is a norm on $\Cat{Met}$. Moreover the metric $\Dist[diam]^+$ induced by this norm is almost the Gromov-Hausdorff distance $\Dist[GH]$; to be precise the identity map 
\[
    (\{\substack{\text{isometry classes of}\\\text{compact metric spaces} }\}, \Dist[GH])
    \to
    (\{\substack{\text{isometry classes of}\\\text{compact metric spaces} }\}, \Dist[diam]^+)
\]
is 2-Lipschitz with Cauchy continuous inverse.

\subsection{Future work}%Arxiv
The next steps in our investigation outline a follows:
\begin{itemize}
    \item 
        generalize \cref{lem:almostExpansiveEndo}---which goes back to Freudenthal and Hurewicz---to categories with a capacity. An assumption will be that for every $\obj \in \cat_0$ there is a map $N \to \operatorname{Sub}(X)$ compatible with $c\colon \operatorname{Sub}(X) \to [0, \infty]$
    \item 
        find various applications thereof, e.g. to the category of metric measure spaces.
    \item 
        starting with a normed category $\cat$ define a normed completion of the category consisting ind-objects and norm-converging pro-ind-limits of morphisms: Before taking a limit we take the under category $T/\cat$---$T$ being the terminal object--- in order to fix a base point. 
        Then objects are defined by some directed index category $\cat[I]$ and a morphism $\vec{\obj}\colon \cat[I] \to T/\cat$ satisfying the Cauchy condition (compare \cite[Def.~3.3]{Kubis17}):
        \[
            \forall \varepsilon > 0\colon
            \exists i_\varepsilon\colon
            \forall i\to i' \text{ with } i \geq i_\varepsilon\colon
            \norm{\obj_{i\to i'} } , \varepsilon
        \text.
        \]
        As a motivation why this Cauchy condition is impose consider, for instance, in the case of the Lipschitz norm \cref{eq:snorm_Lip_teaser} that the diagram $\obj_n = ([0,1], \frac{1}{n} \Dist[$[0,1]$] ), \obj_{n\to m} = \id_{[0,1]} 
        $ should be ruled out as an object in the completion since this would be only a pseudometric space leading to division by zero in \cref{eq:snorm_Lip_teaser}.
        
        Finally, the set of morphisms between $\vec\obj\colon \cat[I] \to T/\cat $ and $\vec{\obj*}\colon \cat[J] \to T/\cat $ is given by all $ f_j^i \in
            \hom{\vec\obj}{\vec{\obj*}} = \lim_i \colim_j \hom{X_i}{Y_j}
        $ such that $ \limsup_{ij} \norm{f_j^i}^{*\mathrm{R}} \leq \limsup \norm{f_j^i} < \infty $.
        This condition corresponds to Kubi\'s's axiom (N3) (cf.\ \cref{rem:Kubis}).
    \item 
        Define a norm on this category by means of a Choquet style integral:
        For a directed set $ I = (I, \leq) $ and an order preserving function $F\colon I\to [0,1]$, thought of as the distribution of a probability measure, set for $f \in \hom{\vec\obj}{\vec{\obj*}} $
        \begin{align*}
                \int f(i)  \,  \mathrm{d} \dot{F}
            &\coloneqq
                \int 1 - F (\sup \setBuilder{i}{f(i) \leq t}) \,  \mathrm{d} t
        \intertext{where}
                f(i) 
            &\coloneqq
                \inf \setBuilder{ \norm{g} }{ g \in \hom{X_i}{Y_j} \text{ with } \iota_{ij} (g) = \proj_i f }
        \end{align*}
        where $ \iota_{ij} $ is the universal map $ \hom{X_i}{Y_j} \to \colim_{j\in I} \hom{X_i}{Y_j} $.
        In the example $ ( \Cat{Met}, \Norm{}{diam} ) $ of metric spaces this corresponds to the pointed Gromov-Haus\-dorff distance.
    \item 
        Generalize the notion of a normed category to $2$-categories: 
        Require the norm to be a 2-morphism and weaken \cref{lst:def:norm1} by requiring only being isomorphic up to 2-morphisms. This generalization should for instance capture coarse structure.
\end{itemize}

\subsection{Versions of this paper}

Changes from version 2 to version 3:
\begin{itemize}
    \item \cref{lem:capacities_Galois} added.
    \item Approach to compact metric spaces and definition of $\Norm{}{diam}$ changed (formerly confusingly denoted by $\Norm{}{dil}$).
    \item \Cref{ssec:automatons}, example of automatons, added.
\end{itemize}

\subsection{Acknowledgements}

The authors thank Paolo Perrone for his hints on the literature.

\section{Definitions}
    \label{sec:def}
\begin{definition}\label{def:seminorm}
    A \definiendum{seminorm} on a category $\cat$ is a function $ \norm{} \colon \cat_1 \to [0,\infty]$ such that
    \begin{enumerate}[label=(N\arabic*)]
        \item\label{lst:def:seminorm1}
            $ \norm{{\id_X}} = 0 $ for every object $X\in \cat_0$,
        \item\label{lst:def:seminorm2} 
            $ \norm{f \comp* g}{}{} \leq \norm{f}{}{} + \norm{g} $
            (triangle inequality).
    \end{enumerate}
    The tuple $ (\cat, \norm{}) $ is called a \definiendum{seminormed category}.
\end{definition}

Note that in the literature a seminorm is often called a norm \cite{Lawvere73,Grandis07}.
Moreover, note that we do not require the obvious strengthening of \cref{lst:def:seminorm1}, namely that the seminorm of every categorical isomorphism vanishes.
An explanation how to view this as a generalization of a seminorm on a vector space is found in \cref{ssec:GrothNorm}.
An isomorphism $f\colon X \to Y $ with inverse $g\colon Y \to X $ is called a \definiendum{norm isomorphism} if $\norm{f} = \norm{g} = 0$. By \cref{lst:def:seminorm2} being norm isomorphic is an equivalence relation.  
Moreover any morphism with norm zero is called a \definiendum{modulator}.
Often the category with objects $\cat_0$ and all modulators of $ \cat_1 $ as morphisms has good categorical properties. We will denoted it by $\boxed{M}(\cat, \norm{})$.
Two seminormed categories are called \definiendum{isomorphic} if there is a norm preserving categorical isomorphism between them.
A seminorm or norm, respectively, induces a seminorm or norm, resp., on the opposite category in the obvious way.
    
\begin{definition}\label{def:norm}
    A seminorm is called a \definiendum{norm} if 
    for all objects $X, Y$ the following holds
    \begin{enumerate}[resume,label=(N\arabic*)]%
    \setcounter{enumi}{2}
        \item\label{lst:def:norm1}
            if there are modulators $f\colon X \to Y$ and $g\colon Y \to X$, then $X$ and $Y$ are norm isomorphic; and
        \item\label{lst:def:norm2}
            if for all $ \varepsilon > 0 $ there is $ f\colon X \to Y $ with $\norm{f}\leq \varepsilon$, then there is a modulator $ f\colon X \to Y $.
    \end{enumerate}
\end{definition}    

The way to view \cref{lst:def:norm1} is that a form of CSB theorem holds. The moral idea is that $\norm{f} = 0$ is a property that is stronger than being monic and $\norm{}$ measures the deviation from this property.

\subsection{Induced norms and distances}
\begin{subequations}
Let 
% $ \sup\nolimits^0 M \coloneqq \sup M\cup\{0\} $ for $M \subseteq [-\infty, \infty]$ and let 
$ (\cat, \norm{}) $ be a seminormed category.
We define the \definiendum{left dual seminorm} as
\begin{align}\label{eq:rightdualseminorm}
        \norm{f}^{*\mathrm{L}}
    \coloneqq
        \supBd_{f'}\left( \norm{f'} - \norm{f'\comp* f}\right)
\text,
\intertext{where $ \smash{X' \xrightarrow{f'} X \xrightarrow{f} Y }$, and the \definiendum{right dual seminorm} as}
\label{eq:leftdualseminorm}
        \norm{f}^{*\mathrm{R}}
    \coloneqq
        \supBd_{f''}\left( \norm{f''} - \norm{f\comp* f''}\right)
\end{align}
where  $ \smash{X \xrightarrow{f} Y \xrightarrow{f''} Y'} $.
\end{subequations}
The seminorm $\norm{}$ is called \definiendum{left reflexive} if $\norm{}^{*\mathrm{L}*\mathrm{L}} = \norm{}$ and \definiendum{right reflexive} if $\norm{}^{*\mathrm{R}*\mathrm{R}} = \norm{}$. 
As opposed to the case of normed spaces, the dual in our case does not define an entirely new category but merely a new norm on the same category.
%This is of course reminiscent of the notion of a reflexive Banach space, in which the evaluation map establishes a natural isomorphism between the given Banach space and its second dual.  Here, though, we merely apply the adjective "reflexive" to the norm in question, not the space carrying that norm. 

To check that the left dual and right dual seminorms are actually seminorms observe for \cref{lst:def:seminorm1} that $ 
    \norm{{\id}}^{*\mathrm{L}} = \supBd \norm{f'} - \norm{f'\comp* {\id}} = 0
    = \supBd \norm{f'} - \norm{{\id}\comp* f'} = \norm{{\id}}^{*\mathrm{R}}
$. 
We show that \cref{lst:def:seminorm2} holds for left duals and then apply that in the opposite category to show that it holds for right duals.  To this end, 
% \cref{lst:def:seminorm2} and $\norm{}^{*\mathrm{L}}$ 
observe that for any diagram $ X \xrightarrow{f} Y \xrightarrow{g} Z $
\begin{align*}
            \norm{f\comp* g}^{*\mathrm{L}}
    &=      \supBd_{h'} \norm{h'} - \norm{ h' \comp* f \comp* g }
\\
   & =      \supBd_{h'} \norm{ h' \comp* f } - \norm{ h' \comp* f \comp* g } + \norm{h'} - \norm{ h' \comp* f }
\\
    &\leq   \supBd_{h'} \norm{h' } - \norm{ h' \comp* f } + \supBd_{h'} \norm{h'} - \norm{ h' \comp*  g }
\\
    &=   \norm{f}^{*\mathrm{L}} + \norm{g}^{*\mathrm{L}}
\text.
\end{align*}
These arguments transfer to the right dual by the fact that the seminorm induced by the right dual on the opposite category coincides with the left dual of norm induced on the opposite category by the original norm.

\begin{remark}\label{rem:Kubis}
    \Textcite{Kubis17} defines a norm in our terminology as a seminorm $\norm{}\colon \cat_0 \to [0,\infty]$ such that $ \norm{}^{*\mathrm{L}} \leq \norm{} $. He defines a completion of a category with respect to such a norm and proves a version of Banach's fixed point theorem in this set-up.
\end{remark}

On the class of norm isomorphism classes of objects of $\cat$, $\skeleton_0(\cat, \norm{}) $,\footnote{Note the foundational remark in the introduction. In many examples $\skeleton_0(\cat, \norm{}) $ admits a set of representatives.}
we define the \definiendum{pqmetric} or \definiendum{pseudoquasimetric induced by $\norm{}$}
\begin{equation}
        \Dist_{\norm{}}(\Hat{X}, \Hat{Y})
    \coloneqq
        \inf\setBuilder{\norm{f}}{f \in\hom X Y \text{ for } X \in \Hat{X}, Y \in \Hat{Y}}
\text{.}
\end{equation}
Note that by \cref{lst:def:seminorm2} for the computation of $\Dist_{\norm{}}(\Hat{X}, \Hat{Y})$ it is sufficient to look at fixed representatives of $\Hat{X}$ and $\Hat{Y}$.
Observe further that this is indeed a pqmetric since $\Dist_{\norm{}}(X, X) \leq \norm{\id_X} = 0$ and the triangle inequality holds by
\begin{align*}
        \Dist_{\norm{}}(X, Z)
    &=
        \inf\setBuilder{\norm{f}}{f\colon X \to Z}
\\
    &\leq
        \inf\setBuilder{\norm{f_1\comp*f_2}}{f_1\colon X \to Y, f_2\colon Y \to Z}
\\
    &\leq
        \inf\setBuilder{\norm{f_1}+\norm{f_2}}{f_1\colon X \to Y, f_2\colon Y \to Z}
\\
    &=
        \Dist_{\norm{}}(X, Y) + \Dist_{\norm{}}(Y, Z)
\end{align*}
for all $X,Y,Z \in \cat_0$.
\begin{subequations}
Symmetrizing in some way gives a pseudometric, e.g.\ by
\begin{align}
    \Dist_{\norm{}}^{\vee}(X, Y) &\coloneqq  \Dist_{\norm{}}(X, Y) \vee \Dist_{\norm{}}(Y, X)
\\
        \Dist_{\norm{}}^{+}(X, Y)   
    &\coloneqq  
        \frac{1}{2} \left(\Dist_{\norm{}}(X, Y) + \Dist_{\norm{}}(Y, X)\right)
\\
        \Dist_{\norm{}}^p(X, Y) 
    &\coloneqq 
        \sqrt[\leftroot{2}\uproot{3} \scriptstyle p]{\Dist_{\norm{}}(X, Y)^p + \Dist_{\norm{}}(Y, X)^p} 
\end{align}
for $p\in [1, \infty)$.
If $\norm{}$ is actually a norm, so $ \Dist_{\norm{}} $ is a metric.
\end{subequations}

\begin{workingNotes}
\subsection{Construction}
\todo{work in progress}
In practise a norm if often initially only defined on a category $\cat $ of which we think of as consisting compact objects. From this category we want to define a norm on maps between spaces that can be exhausted by compact spaces. In categorical terms the exhaustion process is captured by forming an ind-object, i.e. a colimit along a directed system.

\begin{definition}\label{def:properSpaces}
    Given a seminormed category $(\cat, \norm{})$ with terminal object $T$ we define the seminormed category $\Cat*{proper-}(\cat, \norm{})$ by the following data:
    \begin{itemize}
        \item Objects are inductive systems $\cat{I} \to T / \cat$, also known as ind-objects or formal filtered colimits \textbf{modulators?}. They are denoted by $\vec{X}, \vec{Y},$ etc.
        \item Morphisms are filtered limits of filtered colimits of hom-sets in $\cat$, i.e.
        \[
          \hom{\vec X}{\vec Y}
          =
          \lim \colim \hom{X_i}{X_j}
        \]
        \item the norm is given by the Choquet integral

        ...
    \end{itemize}
\end{definition}
\todo{Remark on Choquet integral}

\begin{proposition}
    The above defined function on $\Cat*{proper-}\cat$ is indeed a seminorm.
\end{proposition}
\end{workingNotes}

\section{Canonical examples}
    \label{sec:examples}
\subsection{Sets}\label{ssec:Sets}

On the category $\Cat{Set}$ of sets we define for a function $f \colon X \to Y $ the norm $\Norm{}{set}$ measuring the deviation of a function from being injective: we set
\begin{equation}
    \label{eq:norm_set_Y}
    \Norm{f}{set} = \log \supBd[1]_{y\in Y} \# f^*(\{y\})
\text.
\vspace{-1mm}
\end{equation}
We check that $ \Norm{}{set} $ is a norm.
For the seminorm properties observe that $\Norm{\id_X}{set} = 0$ for any set $X$. 
Moreover the triangle inequality is satisfied as it holds trivially whenever $\Norm{f}{set} = \infty$ or $ \Norm{g}{set} = \infty$ and otherwise---using \cref{eq:norm_set_Y}---
\begin{align*}
        \Norm{f\comp* g}{set}
    &=   \log\supBd[1]_{z\in Z} \# (f\comp*g)^*\{z\}
\\
    &=   \log\supBd[1]_{z\in Z} \# \bigcup \setBuilder*{f^*(\{y\})}{y \in g^* \{z\}}
\\
    &\leq
        \log \left(\supBd[1]_{y\in Y} \# f^*(\{y\}) \cdot \supBd[1]_{z\in Z} \# g^* \{z\}\right)
\\
    &=
        \Norm{f}{set} + \Norm{g}{set}
\text.
\end{align*}
Hence $\Norm{}{set}$ is a seminorm.
We continue with the norm properties: $\Norm{f}{set} = 0$ is to say that $f$ is injective.
Property \cref{lst:def:norm1} is exactly the Schröder-Bernstein theorem.
The property \cref{lst:def:norm2} is trivial since the image of $\Norm{}{set}$ is discrete.

\subsection{Simplicial complexes}

We can use the same norm $\Norm{}{set}$, from the previous subsection, on simplicial complexes.
Recall that an (unoriented) simplicial complex on a set $V$ is a pair $(V, X)$ where $X$ is a subset of $\powerSet(V)$ such that
\begin{enumerate}
    \item each $s\in X$ is finite and nonempty,
    \item $\{x\}\in X$ for each $x \in V$, and
    \item $s' \in X$ for all $s' \subseteq s \in X$ and $s'\neq \emptyset$.
\end{enumerate}
The elements of $V$ are called vertices.
A set $s\in X$ is called simplex, and $n$-simplex if $s$ has $n$ elements. Further let $X_n$ denote the set of all $n$-simplices in $X$. A simplicial complex is called finite whenever it is finite as a set.
A morphism of simplicial complexes is a function $f\colon V \to W$ such that $(f_*)_*(X) \subseteq Y$.
Let $\Cat{SimpCplx}$ denote the category of simplicial complexes and morphisms of simplicial complexes. Let $\Cat{FinSimpCplx}$ denote the full subcategory of finite simplicial complexes, i.e.\ simplicial complexes with a finite set of vertices (or equivalently a finite set of simplices).

\begin{proposition}
    The seminormed category $ (\Cat{FinSimpCplx}, \Norm{}{set}) $ is normed.
\end{proposition}

\begin{proof}
    The axiom \cref{lst:def:norm2} is trivial since the image of $\Norm{}{set}$ is discrete.
    For \cref{lst:def:norm1} assume that $(V, X)$, $(W, Y)$ are two simplicial complexes and $ f\colon V \to W $, $g\colon W \to V$ morphisms of simplicial complexes with $\Norm{f}{set} = \Norm{g}{set} = 0$. Thus $ f $ and $g$ are injective. Hence $f_*$ and $g_*$ are injective. This is to say that both functions map $n$-simplices to $n$-simplices for every $n$. Hence both complexes have the same number of $n$-simplices for every $n$. Since for each $n$ the set of simplices $X_n$ is finite and $f_*$ is injective, the map $f_*$ is actually bijective. Hence both simplicial complexes are norm isomorphic.
\end{proof}

\subsection{Automatons}
    \label{ssec:automatons}
By an automaton we understand a triple $M = (S, A, \delta)$ where $S$ and $A$ are sets and $\delta$ is a relation on $S \times A \times S$. The set $S$ is interpreted as states; $A$ as action and $\delta$ as a (non-deterministic) transition rule.

This gives naturally rise to the following normed category:
\begin{align*}
    \cat[M]_0 &\coloneqq S
\text;
\\
    \hom[M]{s}{t} &\coloneqq \setBuilder*
                        {\text{words } a_1\ldots a_n \text{ in } A}
                        { \substack{ \exists s_0, s_1, \ldots, s_n \in A\colon \\ (s_0, a_1, s_1), \ldots, (s_{n-1}, a_n, s_n) \in \delta \\ \text{ and } s_0 = s, s_n = t }  }
\text,
\intertext{note that $\hom[M]{s}{s}$ always contains the empty word by choice $s_0 = s$;}
    a_1\ldots a_n \comp* b_1\ldots b_m &\coloneqq a_1\ldots a_n b_1\ldots b_m
\text;
\\
    \norm{a}_M &\coloneqq \operatorname{length}(a)
\text.
\end{align*}

\subsection{Cost functions and polarization}
        \label{ssec:metricSp} 
    Let $c$ be a \definiendum{cost function} on $\spSet$, i.e.\ a map $c\colon \configurations[2] \spSet \to [0,\infty] $ where the configuration space $\configurations[2]( \spSet) $ is defined as $ \setBuilder{(x,y) \in \spSet^{\times 2} }{x\neq y} $. Cost functions form the corner stone of transportation theory \cite{Villani08}. Remember that a square-free word is a word in  which no pattern of  the form $xx$  occurs. Define the normed category $(\cat[\spSp], \norm{}_c)$ by
    \begin{align*}
        \cat[\spSp]_0 &\coloneqq \spSet,
    \\
        \cat[\spSp]_1 &\coloneqq \{\text{square-free words over the alphabet }\spSet \},
    \\
        \hom[\spSp]{x}{y} &\coloneqq \setBuilder{w\in \cat[\spSp]_1}{ w \text{ starts with $x$ and ends with $y$} }
    \\
            (x\xi_1\ldots \xi_n y) \comp* (y \eta_1 \ldots \eta_m z)
        &\coloneqq  
            (x\xi_1 \ldots \xi_n y \eta_1 \ldots \eta_mz)
    \\
        \id_x &\coloneqq (x),
    \\
        \norm{(\xi_1\ldots \xi_n)}_c
        &\coloneqq  
            c(\xi_1, \xi_2) + \ldots + c(\xi_{n-1}, \xi_n)
    \end{align*}
    for $(\xi_1\ldots \xi_n)\colon x \to y$.
    This obviously defines a seminorm since the triangle inequality \cref{lst:def:seminorm2} is even an equality.
    
    Moreover this construction induces a qpmetric on $\spSet$, namely $\Dist_c \coloneqq \Dist_{\norm{}_c}$.
    This is automatically a pmetric since any path from $x$ to $y$ can be transformed into a morphism from $y$ to $x$ of the same length by reversing.
    We have $\Dist_c \leq c$ and equality holds if and only if the extension of $c$ to $ \spSet  \times \spSet $ by 0 is a metric. In other words, $\Dist_c$ is the largest pseudometric bounded by $c$.
    %The reader should note that this extension by zero yields that length one words have norm zero.  

\begin{exerciseSolved}
    Now assume that $\spSp = (\spSet, \spMet)$ is a metric space. Then $(\cat[\spSet], \norm{}_c)$ with $ c(x,y) = \spMet(x,y) $ is a normed category: for \cref{lst:def:norm1} observe that a morphism $w = (\xi_1\ldots \xi_n)\colon x \to y$ has  vanishing norm only if $
        \dist{x}{\xi_1} = \dist{\xi_1}{\xi_2} = \ldots = \dist{\xi_{n-1} }{\xi_n} = \dist{\xi_n}{y} = 0
    $ and, hence, $x= \xi_1 = \ldots = \xi_n = y$. Hence $w$ is a word of length one. For \cref{lst:def:norm2} observe that the norm of any morphism $x\to y$ is bounded from below by $\dist{x}{y}$.
    All norm isomorphism classes are singletons.
    
    The left and right duals of $\norm{}_c$ vanish. Both cases are parallel. In case of the right dual---for instance---observe that any composition $f\comp*f'$  has norm at least  $\norm{f'}$. Hence the norm $\norm{}_c$ is only reflexive in the trivial case that $\spMet$ is the vanishing distance.
\end{exerciseSolved}    
    
\begin{exercise}
\begin{subenvironments}
    \begin{xca}
        Assume that $\spSp = (\spSet, \spMet)$ is a metric space. Show that $(\cat[\spSet], \norm{}_c)$ with $ c(x,y) = \spMet(x,y) $ is a normed category
    \end{xca}
    \begin{xca}
        Let $c$ be a cost function. Show that $ \norm{}_c^{*\mathrm{L}} = \norm{}_c^{*\mathrm{R}} = 0 $.
    \end{xca}
\end{subenvironments}
\end{exercise}
    
\begin{Banach}
    Assume, in addition to being a metric, that $\spMet$ is complete.
    Let $\cat[I]$ be a directed set.
    Consider an inductive system $ D\colon \cat[I] \to \cat[\spSp]$ or projective system $D\colon \cat[I]^{\textnormal{op}} \to \cat[\spSp] $.
    Actually, both cases are the same as any word $(\xi_1\ldots\xi_n)$ can be considered as a word from $x$ to $y$ or from $y$ to $x$ without effecting its norm.
    Assume that the inductive system $D$ Cauchy converges.
    Hence we can find a sequence of morphisms $D_1(i_n, j_n)$ such that for all $n$
    \begin{itemize}
        \item $\norm{D_1(i_n, j_n)} < \nicefrac1n$,
        \item $\norm{D_1(i, j)} < \nicefrac1n $ for all $j_n \leq i \leq j$, and
        \item $ j_n \leq i_{n+1} $.
    \end{itemize}
    This is to say that the sequence $D_0(j_n)$ in the metric space $\spSp$ Cauchy converges as $\dist{D_0(j_n)}{D_0(j_{n+m})} < \nicefrac1n$.
    The limit point $x$ of this sequence together with morphisms $\emptyset\colon x \to D_0(i)$ is a weak colimit and weak limit. Hence $\cat[\spSp]$ is a Banach category.
\end{Banach}

\subsection{Grothendieck norm for monoids}\label{ssec:GrothNorm}
Let $M = (M, +, 0)$ be a (not necessarily commutative) monoid.
Further let $\norm{}_M$ be a seminorm on $M$, i.e.\ a map $M\to [0,\infty]$ such that $\norm{a+b} \leq \norm{a} + \norm{b}$.
In the spirit of the previous example we define the category $\cat[M]$ and the \definiendum{Grothendieck seminorm} $\norm{}$ on $\cat$ by
\begin{align*}
    \cat[M]_0 &\coloneqq M
\\
        \hom{a}{b}[\cat[M]] 
    &\coloneqq 
        \setBuilder
            {(f_+, f_-) \in M^{\times 2}}
            {f_+ + a = b + f_-}
\\
        (f_+, f_-) \comp* (g_+, g_-)
    &\coloneqq
        (g_+ + f_+, g_- + f_-)
\intertext{for $
    (f_+, f_-)\colon a \to b$ and $
    (g_+, g_-)\colon b \to c
$ (well-defined since $
        g_+ + f_+ + a 
    =   g_+ + b + f_- 
    =   c + g_- + f_-
$)}
        {\id_a}
    &\coloneqq
        (0, 0)
\qquad\text{for all } a \in \cat[M]_0
\\
    \norm{(f_+, f_-)} &\coloneqq \norm{f_+}_M + \norm{f_-}_M
\text{.}
\end{align*}
This defines indeed a seminorm since $
        \norm{(f_+, f_-)\comp* (g_+, g_-)} 
    =   \norm{(g_+ + f_+, g_- + f_-)}
    \leq\norm{g_+}_M+\norm{f_+}_M+\norm{g_-}_M + \norm{f_-}_M
    =   \norm{(f_+, f_-)} + \norm{(g_+, g_-)}$.

\begin{exerciseSolved}
    If $\norm{}_M$ is a norm (i.e. $\norm{a}_M = 0$ if and only if $a=0$) then $ (\cat[M], \norm{}) $ is a normed category:
    For all $ (f_+, f_-) \in \cat[M]_1 $ we have that  $ \norm{(f_+, f_-)} = 0$ implies $ f_+ = f_- = 0$. Thus in this case  $\norm{}$ fulfils \cref{lst:def:norm1} on the category $\Cat{m}$.
    For \cref{lst:def:norm2} observe that $ \norm{f} < \varepsilon $ implies that $ \norm{a-b} = \norm{f_- - f_+} \leq \norm{f_-} + \norm{f_+} < \varepsilon $.
    Thus is for $a$ and $b$ and every $\varepsilon > 0$ there is $f\colon a \to b$ with $ \norm{f} < \varepsilon $, then $\norm{a-b} = 0$, thus $a = 0$.
    Since $a$ and $b$ we arbitrary, $\norm{}$ fulfills \cref{lst:def:norm2}. Thus $\norm{}$ is a norm.
\end{exerciseSolved}

\begin{exercise}
    \begin{xca}
        Assume that $\norm{}_M$ is a norm (i.e. $\norm{a}_M = 0$ if and only if $a=0$).
        Prove that $ (\cat[M], \norm{}) $ is a normed category.
    \end{xca}
\end{exercise}

\begin{proposition}\label{prop:GrothNorm_Gr}
    Assume that $G$ is a group with a seminorm $\norm{}_G$. Then
    \begin{claims}
        \item\label{claim:prop:GrothNorm_Gr_isom}
            a normed category canonically isomorphic to $(\cat[G], \norm{} )$ is defined by
            \begin{align*}
                \cat[G]_0' &\coloneqq G,
            \\
                \hom{a}{b}[\cat[G]'] &\coloneqq G,
            \\
                f\comp*' g &\coloneqq g - b + f,
            &&
                \text{for } f\colon a \to b \text{ and } g\colon b \to c,
            \\
                \norm{f}' &\coloneqq \norm{f}_G + \norm{-b+f+a}_G 
            &&
                \text{for } f\colon b \to c
            \text{;}
            \end{align*}
        \item\label{claim:prop:GrothNorm_Gr_met}
            assuming that $\norm{a}_M = \norm{-a}_M $, we have $\Dist_{\norm{}'}(a,b) = \norm{-b+a}_G$;
        \item\label{claim:prop:GrothNorm_Gr_duals}
             under the same assumption we have $\norm{}^{*\mathrm{L}} = \norm{}^{*\mathrm{R}} = \norm{}$.
    \end{claims}
\end{proposition}

\begin{exercise}
    \begin{xca}Prove this theorem.\end{xca}
\end{exercise}

\begin{exerciseSolved}
    \begin{proof}
        For \cref{claim:prop:GrothNorm_Gr_isom} we will show that the isomorphism and its inverse are given by the reparameterization functors $ F$ and $ \bar F $ that are the identity on the objects set $G$ and on morphisms given by the assignments
        \[
            F_1       \colon (f_+, f_-)   \mapsto f_+             \quad\text{and}\quad 
            \bar F_1  \colon f            \mapsto (f, -b + f + a)
        \]
        Both maps are inverse to each other as $ f_- = - b + (f_+ + a) = \bar F(f_+) $.
        Hence we only have to check compatibility of $F$ and $ \comp*' $:
        $ 
                F(f_+, f_-) \comp*' F(g_+, g_-) 
            =   g_+ + f_+
            =   F( (f_+, f_-) \comp* (g_+, g_-) )
        $ for $ 
            a \xrightarrow{(f_+, f_-)} b \xrightarrow{(g_+, g_-)} c
        $.
        This concludes the proof that $ \cat[G]' $ defines a category isomorphic to $ \cat[G] $.
        The equality of the norms $\norm{}$ and $\norm{}'$ follows directly from the identity $f_- = - b + (f_+ + a)$.
        
        For \cref{claim:prop:GrothNorm_Gr_met} observe that for every two objects $a, b \in G$ we have $ 
            \Dist_{\norm{}'}(a,b) \leq \norm{0} + \norm{-b + 0 +a} = \norm{-b+a}
        $ and $
            \norm{f} + \norm{-b+f+a} \geq \norm{-f -b+f+a} = \norm{-b+a}
        $ for all $f,g\in G$. Hence $\Dist_{\norm{}'}(a,b) = \norm{-b+a}_G$.
        
        For \cref{claim:prop:GrothNorm_Gr_duals} observe for the left dual that for $f' = (f'_+, f'_-)$ we have $ 
                \norm{f'} - \norm{f'\comp* f} 
            =   \norm{f'\comp* f\comp* (-f)} - \norm{f'\comp* f} 
            \leq\norm{f'\comp* f} + \norm{-f} - \norm{f'\comp* f}
            =   \norm{f}
        $ and $ \norm{-f} - \norm{(-f) \comp* f} = \norm{f} $.
        Hence $\norm{}^{*\mathrm{L}} = \norm{}$.
        By the parallel argument $\norm{}^{*\mathrm{R}} = \norm{}$.
    \end{proof}
\end{exerciseSolved}

\Cref{claim:prop:GrothNorm_Gr_met} of \cref{prop:GrothNorm_Gr} is the motivation to consider a seminorm on a category a generalization of a seminorm on a vector space where the underlying abelian group takes the role of $G$.

\begin{example}[Word metric on a group]
    As an example of a norm on a monoid take generating set $\{g_i\}_{i\in I}$ of a group $G = (G, +, -(\blank), 0)$. Then the free monoid $M$ generated by $\bigcup_{i\in I} \{g_i, -g_i\}$ has the evaluation function $\operatorname{ev}\colon M \to G$. On $M$ there is a norm given on $m\in M$ by the minimal word length of some word $w$ over the alphabet $M$ such that $\operatorname{ev}(w) = m$.
    Word metrics are a fundamental tool in geometric group theory \cite{Loh17}.
\end{example}

\subsection{Operators and normed vector spaces}
    \label{ssec:NVect}

\begin{subequations}
Let $\Cat{NVect}_{\R}$ denote the category of normed vector spaces over the reals with linear maps as morphisms.
Define the norm of a linear map $A \colon V \to W$ as
\[
        \Norm{A}{op}
    \coloneqq
        \log \supBd[1]_{v\in V} \frac{\norm{v}_V}{ \norm{Av}_W }
\text{.}
\]
The pair $(\Cat{NVect}_{\R}, \norm{.})$ is a seminormed category.
Being norm isomorphic is equivalent to being isometric as linear normed spaces.
\end{subequations}

It is not a normed category and especially does not satisfy the CSB axiom \cref{lst:def:norm1}. 
\begin{exerciseSolved}
For an example consider the vector space $\Hoeld{0}([0,1])$ of real valued continuous functions on the unit interval $[0,1]$ vanishing at 0 and 1 and endow this space with the supremum norm $\norm{}$. Define the spaces
\begin{align*}
    V &\coloneqq \Hoeld{0}([0,1])
\\
        W 
    &\coloneqq 
        \Hoeld{0}([0,1]) \oplus \setBuilder{f\in \Hoeld{0}([0,1])}{f \text{ smooth}}
\end{align*}
where $W$ is endowed with the norm $\norm{(f,g)}_c \coloneqq \norm{f} \vee \norm{g}$. The space $W$ is not complete, but $V$ is, so these spaces are not norm isomorphic. But we can find expansions in both directions:
\begin{align*}
    f&\colon V \to W,
&
    f &\mapsto (f,0)
\\
    g&\colon W \to V,
&
        (f,g) 
    &\mapsto 
      \left(
          x 
        \mapsto 
          \begin{cases}
            f(2x)   & x \leq 1/2 \\
            g(2x)   & x \geq 1/2 
          \end{cases}
      \right)
\text{.}
\end{align*}
\end{exerciseSolved}
\begin{exercise}
\begin{xca}
    Find an example that violates \cref{lst:def:norm1}. Hint: take the vector space $\Hoeld{0}([0,1])$ of real valued continuous functions on the unit interval $[0,1]$ vanishing at 0 and 1 and endow this space with the supremum norm $\norm{}$. Define $ V \coloneqq \Hoeld{0}([0,1]) $ and $ W \coloneqq 
        \Hoeld{0}([0,1]) \oplus \setBuilder{f\in \Hoeld{0}([0,1])}{f \text{ smooth}}$.
\end{xca}
\end{exercise}

It's even not possible to ensure \cref{lst:def:norm1} by restricting to the category of Banach space, i.e.\ complete normed vector spaces. Corresponding examples are more complicated but one was found in a celebrated result by \textcite{Gowers96}.
On the other hand the fully faithful subcategory $\prescript{\textnormal{Hilb}}{}{\Cat{NVect}}_{\R}$ of $\Cat{NVect}_{\R}$ consisting of Banach spaces that admit an inner product (i.e.\ admit the structure of a Hilbert space).
Recall that the Hilbert dimension of a Hilbert space is defined as the cardinality of a basis. A basis is by definition a maximal orthonormal set $ E \subset V $, i.e.\ $\langle e, e' \rangle = 0$ for all $e,e'\in E$ with $e\neq e'$ and $\norm{e} = 1$ for all $e \in E $. Especially, $E$ is linearly independent.
Hilbert spaces are norm isomorphic if and only if they have the same Hilbert dimension \cite[Theorem~5.4]{Conway94}.

If there is an expansive operator $A\colon V \to W$ than the dimension of $W$ is not smaller than the dimension of $V$: if $E$ is a maximal orthonormal set in $V$ then $A(E)$ is still linearly independent due to linearity and injectivity of $A$. Since $W$ is a Hilbert space there is a decomposition $W \simeq \overline{A(V)} \oplus W' $. By the Gram-Schmidt process we can find a basis for $\overline{A(V)}$. Extending this basis to a basis of $W$ be see that the dimension of $W$ is not smaller than the dimension of $V$.
Thus if we have expansive operators in both direction, $V$ and $W$ are of the same dimension and hence there is a norm isomorphism.

\begin{exerciseSolved}
The left dual of $\Norm{}{op}$ is the re-scaled operator norm $\log \supBd[1]_{v\in V} \frac{\Norm{Av}W}{\Norm{v}V}$ since for any $v$ with $\norm{Av}_c \neq 0$ one can consider the re-scaled embedding of the one-dimensional subspace $f'\colon \R v \to V, v' \mapsto \frac{1}{\norm{Av}} v' $.
Repeating this argument shows that $\norm{}$ is left reflexive.
\end{exerciseSolved}
\begin{exercise}
\begin{xca}
    Show that the left dual of $\Norm{}{op}$ is the re-scaled operator norm $\log \supBd[1]_{v\in V} \frac{\norm{Av}}{\norm{v}}$ and that $\Norm{}{op}$ is left reflexive.    
\end{xca}
\end{exercise}

\section{Norms from capacities}
    \label{sec:capacities}
For concrete categories $\cat$ most seminorms arise from a function on the subobjects of objects in $\cat_0$---or an extension of this concept---valued in the extended real numbers, called precapacity.
We will define precapacities as a function on subobjects of objects in $ \cat$, the category to be given a seminorm.
Given an object $X$ in a category $\cat$ the slice category $ \cat/X $ is defined as the category with morphisms $ B \in \hom{Y}{X} $ (for any $Y \in \cat_0$) as objects and commuting diagrams
\[
    \begin{tikzcd}
            \source B  \ar[rr, "\varphi"]\ar[dr, "B"] &  & \source C \ar[dl, "C"] \\
                                        & X &
    \end{tikzcd}
\]
as morphisms. Composition is defined by composition of morphisms $\varphi\comp*\psi $:
\[
    \begin{tikzcd}
        \source B  \ar[r, "\varphi"]\ar[dr, "B"]  & \source C \ar[d, "C"] \ar[r, "\psi"] & \source D \ar[dl, "D"] \\
                                    & X &
    \end{tikzcd}
\mathrlap{\text.}
\]

We repeat the standard notion of subobjects from category theory \cites[11]{MacLane94}[A.1.3]{Johnstone02}.
A subobject of an object $X$ in a category $\cat$ is an equivalence class of monomorphisms to $ X $, where equivalent means isomorphic in $\cat/X$. 
We denote the set of such equivalence classes by $\subObjs_0(X)$.
Note that the composition of two monomorphisms is a monomorphism.
\begin{exercise}
\begin{subenvironments}
\begin{xca}
    Prove that two objects $B, C \in (\cat/X)_0$ are isomorphic if $B,C$ are monomorphisms in $\cat$ and there are morphisms $\varphi\in\hom{B}{C}[\cat/X]$ and $\psi\in\hom{C}{B}[\cat/X]$ such that $\varphi$ and $\psi$ are monomorphisms in $\cat$.
\end{xca}
\begin{xca}
    Prove that $\subObjs_0(X)$ becomes a partially ordered set by the relation $\subseteq$ where $B \subseteq C$ if and only if there is a morphism from $B$ to $C$ in $\cat/X$ that is a monomorphism in $\cat$.
\end{xca}
\end{subenvironments}
\end{exercise}

\begin{exerciseSolved}
Two objects $B, C \in (\cat/X)_0$ are isomorphic if there are morphisms $\varphi\in\hom{B}{C}[\cat/X]$ and $\psi\in\hom{C}{B}[\cat/X]$ such that $\varphi$ and $\psi$ are monomorphisms in $\cat$: 
assume that there are morphisms $\varphi\in\hom{B}{C}[\cat/X]$ and $\psi\in\hom{C}{B}[\cat/X]$ such that $\varphi$ and $\psi$ are monomorphisms in $\cat$.
Then, in $\cat$, we have $ C = {\id} \comp* C $ and $ C = \varphi \comp* \psi \comp* C $. 
Hence $ {\id} = \varphi \comp* \psi $. By the parallel argument we get $ {\id} = \psi \comp* \varphi $. This shows that $\varphi, \psi$ are isomorphisms in $\cat$.
Further the property of admitting a monomorphism from $B$ to $C$ gives a partial order on $\subObjs_0(X)$, written $\subseteq$:
\begin{enumerate}
    \item antisymmetry was just proven above,
    \item transitivity means that the composition of monomorphisms is a monomorphism, and
    \item reflexivity means that the identity is a monomorphism.
\end{enumerate}
\end{exerciseSolved}

We will denote by $\subObjs(X)$ the set $\subObjs_0 (X)$ endowed with the partial order $\subseteq$.
From category theory, the pullback\footnote{We denote pullback diagrams by \begin{tikzcd}[ampersand replacement=\&]
    {}\ar[d] \ar[r] \arrow[dr, phantom, "\usebox\pullback" , near start]
 \&{} \ar[d] \\
    {} \ar[r]   \& {}
\end{tikzcd}. } of a monomorphism $C$ along a morphism $f$ is a monomorphism again; it is denoted by $f^* C$:
\[
\begin{tikzcd}
    \source (f^* C) \ar[d, "f^* C"] \ar[r] \arrow[dr, phantom, "\usebox\pullback" , very near start]
 &\source C \ar[d, "C"] \\
    X \ar[r, "f"]   & Y
\end{tikzcd}
\mathrlap{\qquad\text{.} }
\]

\begin{definition}
    A \definiendum{concrete category} $(\cat, F)$ is a category $\cat$ together a faithful functor $\cat \to \Cat{Set}$. The functor $F$ is called \definiendum{forgetful functor}.
\end{definition}

\begin{definition}\label{def:concreteCat_subObj}
    By a \definiendum{concrete category with generalized subobjects $(\cat, F;\linebreak[1] S, F_S, \operatorname{GS} ) $}, which by abuse of notation we will often denote by $(\cat; \operatorname{GS} )$, we understand a concrete category $(\cat, F)$ additionally endowed with
    an extension of the concrete category $ (\cat, F) $, meaning a commutative triangle of functors
            \begin{center}
                \begin{tikzcd}
                        \cat[SC] \ar[dr, "F_S"]        &        \\
                        \cat\ar[u, "S"]\ar[r, "F"]  &   \Cat{Set}
                \end{tikzcd}
            \text,
            \end{center}
    and a selection function $\operatorname{GS} $ that assigns to each object $X \in \cat_0$ a family of subobjects in $\subObjs(S X)$, called \definiendum{(generalized) subobjects},
    such that
    \begin{enumerate}
        \item\label{lst:def:def:concreteCat_subObj_preserve}
            for each $X \in \cat_0$ the order preserving induced functor
            \[
              |\operatorname{GS}|(X)\colon \operatorname{GS}(X) \to \subObjs(F (X)),\quad
              C \mapsto (F_S)_1(C)
            \]
            is well-defined, i.e.\ $ F_S(C) $ is a monomorphism again, and full, i.e.\ $\left|B\right| \subseteq \left|C\right|$ implies $B \subseteq C$ where
            \begin{equation}\label{eq:extension_subobj}
                |C| \coloneqq (F_S C) (\source C) \subseteq F(X)
            \end{equation}
    for any $C \in \operatorname{GS} (X)$ with $X \in \cat_0$ (note that \cref{eq:extension_subobj} is independent of the representative in $\Cat{Set}_1$).
        \item\label{lst:def:def:concreteCat_subObj_preimage}
            if $f\colon X \to Y$ and $C \in \operatorname{GS}(Y)$, then there is a $B \in \operatorname{GS}(X)$ with $ |B| = (F f)^* (|C|) $, that is maximal in $\operatorname{GS}(Y)$ with this property.
            This generalized subobject is called the \definiendum{preimage of $C$ under $f$} and written $ B = f^*(C) $.
    \end{enumerate}
\end{definition}

Often we will encounter the case $\cat[SC] = \cat$ and $S = \id_{\cat}$.
By abuse of notation we will often write $B \in \operatorname{GS}(X)$ for a representative of $B$.
Note that a concrete category with generalized subobjects has some similarity to a Grothendieck topology: think of a representative of an element of $\operatorname{GS}(X)$ as singleton set. Then one may compare the assignment $ \operatorname{GS}$ to a Grothendieck topology, that assigns to each object a family of sieves.

An example of a concrete category with subobjects is given by 
\begin{equation}\label{eq:enoughSubObj_closedSubobjects}
    (\Cat{Top}, F; {\id_{\Cat{Top}}}, F, {\closed} )
\end{equation}
where 
$\closed(\situs)$ is the collection of equivalence classes of homeomorphisms onto closed subspaces of $\situs$ and $F$ is the canonical forgetful functor $\Cat{Top} \to \Cat{Set}$:
the property of having enough subobjects follows immediately from the fact that preimages of closed sets under continuous maps are closed.

\begin{definition}
    A \definiendum{precapacity} $w$ on a concrete category with subobjects $ (\cat; \operatorname{GS} ) $ is a function
    \[
            c\colon \bigsqcup_{X \in \cat[SC]_0} \operatorname{GS}(X)
        \to 
            [-\infty, \infty]
    \]
    and it is called a \definiendum{capacity} if it is monotone or antimonotone, i.e.\ for any two subobjects $A, A' \in \operatorname{GS} (X) $ with $ A \subseteq A' $ we have that $c(A) \leq c(A') $ or $c(A') \leq c(A)$, resp.
\end{definition}

In practise, capacities are often non-negative.
Each precapacity $c$ gives rise to an assignment
\begin{equation}
\label{eq:seminorm_by_w}
        \norm{f}_c
    \coloneqq
        \supBd
            \setBuilder*
                {c  (f^* B)  - c(B)}
                {\substack{B \in \operatorname{GS} (Y), \\ c(B) < \infty}}
\end{equation}
where $-\infty - (-\infty) \coloneqq 0$.
For a concrete category with enough subobjects it is called the \definiendum{seminorm induced by} $w$. The seminorm properties are checked immediately by
\[
        \norm{\id_X}_c 
    =
        \supBd_B c ( \id_X^* (B)) - c(B)
    =
        \supBd_B c(B) - c(B)
    =
        0
\]
and for any diagram $X \xrightarrow{f} Y \xrightarrow{g} Z$ in $\cat$
\begin{align*}
        \norm{f\comp*g}_c
    &=
        \supBd_B
            \left(c (f\comp*g)^* B  - c(B)\right)
\\
    &=
        \supBd_B
            \left(c (f\comp*g)^* B  - c (g^* B) + c (g^* B) - c(B)\right)
\\
    &\leq
        \supBd_B \left(c (f\comp*g)^* B  - c g^* B \right) + 
        \supBd_B    \left(c g^* B - c(B)\right)
\\
    &=
        \norm{f}_c + \norm{g}_c
\text{.}
\end{align*}
%\todo{Think about this argument and the fact that it imposes no conditions on precapacities in order to obtain seminorms.}

\begin{example}
As an example we will explain how to formulate the seminormed category $(\Cat{NVect}_{\R}, \Norm{}{op})$ from \cref{ssec:NVect} in the framework of capacities explained above.
Set
\begin{align*}
        (\Cat{SNVect}_{\R})_0
    &\coloneqq
        \setBuilder{(F,V)}{ F\subseteq V, V\in \Cat{NVect}_{\R} }
\\
        \hom{(F,V)}{(G,W)}[\Cat{SNVect}_{\R}]
    &\coloneqq
        \setBuilder{f\in \hom[Set]{F}{G} }{ \exists A \in (\Cat{NVect}_{\R})_1\colon A|_F = f }
\\
        S
    &\colon 
        \Cat{NVect}_{\R}\to\Cat{SNVect}_{\R}
\\
        S(V)
    &\coloneqq
        (V,V)
\\
        F_S(F, V)
    &\coloneqq
        F
\\
        \operatorname{GS} (V)
    &\coloneqq
        \setBuilder[\big]{(F,V) \xrightarrow{ \smash{(v,w)\mapsto (v,w) } } (V, V) }{F \subseteq V}
\\
        c( F, V )
    &\coloneqq
        \sup_{v\in F}\log\norm{v}_V
\end{align*}
where $V = (V, \norm{v}_V)$. This actually defines a category with subobject as both properties \cref{lst:def:def:concreteCat_subObj_preserve,lst:def:def:concreteCat_subObj_preimage} are obvious. 
\begin{exerciseSolved}
Moreover for any $A\colon V \to W$
\begin{align*}
        \norm{A}_c 
    &=
        \supBd_{\substack{G \subseteq W, \\ \norm{G}_c< \infty}} 
            c (f^*G,V)  - \log\norm{G}_c
\\
    &=
        \supBd_{\substack{G \subseteq W, \\ \norm{G}_c< \infty}} 
            \Big(\sup_{v\in f^*G}\log\norm{v}_V\Big)  - \Big(\sup_{w\in G}\log\norm{v}_W\Big) 
\end{align*}
Thus we have $
    \norm{A}_c \geq \supBd_{w \in W} (\sup_{v\in V, f(v) = w} \norm{v}_V) - \norm{w}_W
$ and also $
    \norm{A}_c \leq \supBd_{\substack{G \subseteq W, \norm{G}_c< \infty}, v \in f^* G} 
            (\log\norm{v}_V)  - \log\norm{f(v)}_W 
$. Hence $\norm{A}_c = \Norm{A}{op}$.
\end{exerciseSolved}
\end{example}

\begin{exercise}
\begin{xca}
     Show that $\norm{A}_c = \Norm{A}{op}$.
\end{xca}
\end{exercise}

\subsection{Capacities by images}

A natural question is, when one can compute capacities given an "image" function $
    g\colon \operatorname{GS} (\obj) \to \operatorname{GS} (\obj*) 
$.
The leading example in this regard are the subobjects given by the power set and the corresponding adjunctions $f_*, f_!\colon \powerSet(X) \to \powerSet(Y)$ characterized by the adjunctions $
  f_* \dashv f^* \dashv f^!
$ (\cref{ssec:setTh}).
As the collections of subobjects are posets, these adjunctions are actually Galois connections.

\begin{lemma}\label{lem:capacities_Galois}
    Let $X$ and $Y$ be two objects in a category with generalized subobjects. Assume either that 
    \begin{enumerate}
        \item $c$ is monotone and there is a direct image $ f_*\colon \operatorname{GS} (\obj) \to \operatorname{GS} (\obj*) $, i.e. a Galois connection $f_* \dashv f^*$, or
        \item $c$ is antimonotone and there is a small image $ f_!\colon \operatorname{GS} (\obj) \to \operatorname{GS} (\obj*) $, i.e. $ f^* \dashv f^! $.
    \end{enumerate}
    Then
    \begin{align*}
            \norm{}_c 
        &= 
            \lambda f. \supBd \setBuilder*
                {c(A) - c(f_* A)}
                {A \in \operatorname{GS} (\obj), c(f_* A) < \infty}, 
        \text{ or }
    \\
            \norm{}_c 
        &= 
            \lambda f. \supBd \setBuilder*
                {c(A) - c(f_! A)}
                {A \in \operatorname{GS} (\obj), c(f_!A) < \infty},
    \text{ resp.}
    \end{align*}
\end{lemma}

\begin{proof}
    Both claims are implied by the following estimate, where $g$ equals $f_*$ or $f_!$:
    \begin{align*}
        \norm{f}_c
        &=      \supBd \setBuilder* {c(f^* B) - c(B)}{B \in \operatorname{GS} (\obj*), c(B) < \infty}
    \\  &\geq   \supBd \setBuilder* {c(f^* B) - c(B)}{B = g(A), A \in \operatorname{GS} (\obj), c(B) < \infty}
    \\  &=      \supBd \setBuilder* {c(f^* g A) - c(g A)}{A \in \operatorname{GS} (\obj), c(g A) < \infty}
    \intertext{we have $c(f^* g A) \geq c(A)$ due to---in the first case---$ f^* f_* A \supset A $ and monotonicity of $c$ and---in the second case---$ f^* f_! A \subset A $ and anti-monotonicity }
        &\geq   \supBd \setBuilder* {c( A) - c(g A)}{A \in \operatorname{GS} (\obj), c(g A) < \infty}
    \\  &\geq   \supBd \setBuilder* {c( A) - c(g A)}{A = f^*(B), B \in \operatorname{GS} (\obj*), c(g A) < \infty}
    \\  &=      \supBd \setBuilder* {c( f^* B) - c(g f^* B)}{B \in \operatorname{GS} (\obj*), c(g f^* B) < \infty}
    \intertext{we have $c(g f^* B) \leq c(B)$ due to---in the first case---$ f_* f^*  B \subset B $ and monotonicity of $c$ and---in the second case---$ f_! f^*  B \supset B $ and anti-monotonicity }
        &\geq   \supBd \setBuilder* {c( f^* B) - c(g f^* B)}{B \in \operatorname{GS} (\obj*), c(B) < \infty}
    \intertext{and due to the same fact also}
        &\geq   \supBd \setBuilder* {c( f^* B) - c(B)}{B \in \operatorname{GS} (\obj*), c(B) < \infty}
    \\  &=      \norm{f}_c
    \qedhere
    \end{align*}
\end{proof}

\subsection{Dual seminorms from capacities}

Let $c$ be a precapacity.
In view of \cref{eq:seminorm_by_w} one is of course tempted to look at the quantity
\begin{equation}
\label{eq:precapacity_OpSNorm}
        \norm{f}_{-c}
    \coloneqq
        \supBd
            \setBuilder*
                {c(B) - c  f^* B}
                {\substack{B \in \operatorname{GS} (Y), \\ c(B) < \infty}}
.
\end{equation}

\begin{subenvironments}
\begin{exercise}
\begin{xca}\label{eq:Rdual_le_precapacity}
    Show that for any precapacity $c$ it holds that $ \norm{f}_c^{*\mathrm{R}} \leq \norm{f}_{-c} $ 
    and therefore $
        \norm{f}_{-c}^{*\mathrm{L}} 
    \geq \norm{f}_c
    \geq \norm{f}_{-c}^{*\mathrm{R}}
$.
\end{xca}
\begin{xca}\label{eq:rightBidual_le_norm}
    Moreover for the \definiendum{left} and \definiendum{right biduals}, i.e.\ the  quantities $\norm{}^{*\mathrm{L}*\mathrm{L}}$ and $\norm{}^{*\mathrm{R}*\mathrm{R}}$, show that $
        \norm{}^{*\mathrm{L}*\mathrm{L}}, \norm{}^{*\mathrm{R}*\mathrm{R}} \leq \norm{}
    $.
\end{xca}
\end{exercise}
\begin{exerciseSolved}
    \begin{subequations}
    It turns out to be related by the inequalities
    \begin{equation}\label{eq:Rdual_le_precapacity}
        \norm{f}_c^{*\mathrm{R}} \leq \norm{f}_{-c}
    \end{equation}
    as is easily checked:
    \begin{align*}
            \norm{f}_c^{*\mathrm{R}}
        &=  \supBd_{f'} \norm{f'}_c - \norm{f\comp* f'}_c
    \\
        &=  \supBd_{f'} \supBd_B c  f^{\prime *} B - c B - \supBd_B c  (f\comp* f')^{*} B - c B
    \\ 
        &\leq \supBd_{f', B} c  f^{\prime *} B - c B - \left( c  (f\comp* f')^{*} B - c B \right)
    \\
        &\leq \supBd_B c   B - c  f^{*} B
    \\
        &=  \norm{f}_{-c}
    \text.
    \end{align*}
    \end{subequations}
    
    For the \definiendum{left} and \definiendum{right biduals} we have that 
    \begin{equation}\label{eq:rightBidual_le_norm}
        \norm{}^{*\mathrm{L}*\mathrm{L}}, \norm{}^{*\mathrm{R}*\mathrm{R}} \leq \norm{}
    \text;
    \end{equation}
    we check this in the case of the left bidual by estimating for any morphism $f$
    \begin{align*}
            \norm{f}^{*\mathrm{L}*\mathrm{L}}
        &=   \supBd_{f'} \norm{f'}^{*\mathrm{L}} - \norm{f'\comp* f}^{*\mathrm{L}}
    \\  &=   \supBd_{f'} 
                    \supBd_{f''} \left( \norm{f''} - \norm{f''\comp* f'} \right) 
                  - \supBd_{f''} \left( \norm{f''} - \norm{f''\comp* f' \comp* f} \right) 
    \\  &\leq\supBd_{f', f''} \left( \norm{f''} - \norm{f''\comp* f'} \right) 
                    -   \left( \norm{f''} - \norm{f''\comp* f' \comp* f} \right) 
    \\  &=\supBd_{f', f''} \norm{f''\comp* f' \comp* f} - \norm{f''\comp* f'}
    \\  &\leq \norm{f}
    \end{align*}
    using triangle inequality in the last step; the case of the right bidual is checked by a parallel argument.
\end{exerciseSolved}
\end{subenvironments}

\subsection{Some crucial seminorms}
\label{ssec:crucialSeminorms}

In this preprint and future work we will study the following examples of capacities and develop simplified characterizations for the induced norms in the following cases.
Note the conventions $\log 0 = -\infty$ and $\abs{-\infty} = \infty$.
\begin{enumerate}
    \item 
        topological dimension $C \mapsto \abs{\log (1+\Dim C)}$ for $\operatorname{SO}(\situsSet) = \powerSet(\situsSet) $ in \cref{sec:top} giving rise to a seminorm measuring deviation of a map from being light.
    \item 
        logarithmic number of connected components
        \[
            \operatorname{SO}(\situsSet) = \closed(\situs)
        \quad\text{and}\quad
            \abs{\log \# {\operatorname{I}}} \colon \situs \mapsto \abs{\log(\# (\operatorname{I} \situs) ) }
        \]
        $ $ for $ $where $ \operatorname{I} \situs $ the set of connected components of $\situs$. In \cref{sec:top} this gives rise to a seminorm measuring deviation of a map from being monotone.
    \item 
        diameter $\diameter \spSp$ of a metric space $\spSp$ for $\operatorname{SO}(\spSp) = \powerSet(\spSp) $ giving rise to a norm, that quasi-metrizes Gromov-Hausdorff convergence.
    \item
        negative logarithmic diameter $-\log \diameter \spSp = - c_{\mathrm{Lip}} $ giving rise to the Lipschitz norm $\Norm{}{Lip} $, that quasi-metrizes Lipschitz convergence of metric spaces. For a map $ f\colon \spSp \to \spSp* $ it may be  expressed as
        \begin{align}\nonumber
                \Norm{f}{Lip}
            &\coloneqq 
                \norm{f}_{-c_{\mathrm{Lip} } }
        \\ \nonumber
            &= 
                \supBd \setBuilder
                    { \log (\diameter B) - \log(\diameter f^* B) }
                    {B \subseteq \spSet*, \log\diameter B > - \infty}
        \\ \label{eq:snorm_Lip_teaser}
            &= 
                \supBd \setBuilder*
                    { \log \frac{ \diameter B }{ \diameter f^* B } }
                    {B \subseteq \spSet*, \diameter B > 0}
        \text.
        \end{align}
        Note that $\diameter B > 0$ implies $ \diameter f^* B > 0 $ for metric spaces.
\begin{Wasserstein}
    \item
        The Wasserstein capacity, see \cref{ssec:snorm_Wasserstein}, for projective metric measure spaces is defined on equivalence classes of metric measure spaces under the equivalence relation $(\spSet, \lambda\spMet, \spVol ) \sim_{\mathrm{W}} (\spSet, \spMet, \lambda \spVol )$.
        By extending the category of such spaces to a category $\Cat{TPmMet}$ of such spaces endowed with a real valued test function we can define a capacity by
        \[
                \capacity[W]( [\spSp], \varphi )
            \coloneqq
                \supBd \setBuilder*
                    { \int \varphi d \spVol }
                    {\substack{(\spSet, \spMet, \spVol) \in \equivClass\spSp{W}, \\
                        \hoeld*{(\spSet, \spMet) \xrightarrow{\varphi} [0,1] }1 \geq 1}  }
        \text.
        \]
    \item
        the \definiendum{total Wasserstein capacity}
        \[
            c_{\mathrm{W}}^* = c_{\mathrm{W}} + \volume
        \text.
        \]
\end{Wasserstein}
\end{enumerate}

\section{Topological spaces}
\label{sec:top}
The subject of this section are two seminorms both measuring the increase of complexity when passing from a subset of the domain to its preimage.
In the first case complexity is measured by topological dimension in the second one by number of connected components.
These norms measure the deviation from being light and monotone respectively.
The well-known monotone-light factorization implies that in the case of compact spaces the sum of both norms, which we call the topological norm $\Norm{}{top}$, is a norm.
Actually, the monotone-light factorization is a strengthening of the norm property \cref{lst:def:norm1}.
The idea of the monotone-light factorization goes back to \textcite{Eilenberg34,Whyburn34} independently.
For a historical overview about the monotone-light factorization and its variations consult \cite{Lord97}.

\subsection{Dimension seminorm}

Recall that a map $f\colon \situs \to \situs*$ is \definiendum{light} if the fiber $f^*y \coloneqq f^*(\{y\})$ is totally disconnected for every $y \in \situsSet*$.
Further we define the \definiendum{dimension seminorm} using the precapacity $ A \mapsto \abs{\log(1+\Dim (A))} $ by
\begin{align*}
        \Norm{f}{dim} 
    &\coloneqq
        \norm{f}_{\abs{\log(1+\Dim)}}
\\
    &=
        \supBd_{\substack{A \in \powerSet (\situsSet*), \\ \Dim A < \infty}} 
             \abs{\log(1+\Dim f^* A)}  - \abs{\log(1+\Dim A)}
\end{align*}
where we used \cref{lst:def:def:concreteCat_subObj_preimage} of \cref{def:concreteCat_subObj} in the last step.

To further simplify this expression we use the Hurewicz formula which states
\begin{equation}\label{eq:dim_ineq}
    \Dim \situs \leq \Dim \situs* + \sup_{y \in Y} \abs{\Dim (f^*\{y\})}
\end{equation}
\cites[Ch.\ 9, Prop.\ 2.6]{Pears75}[VI.4]{HurewiczWallman48}
for any continuous closed surjection $f$ from a $\mathrm{T}_4$-space $\situs$ to a metrizable space $\situs*$.
A direct consequence of this formula is that the function $\lambda A. \abs{1 + \Dim A}$ is monotone with respect to inclusion. 
So this function is indeed a capacity.
\begin{subequations}
Hence \cref{lem:capacities_Galois} implies 
\begin{equation}
        \Norm{f}{dim} 
    =
       \supBd_{\substack{A \in \powerSet (\situsSet), \\ \Dim f_* A < \infty}} 
             \abs{\log(1+\Dim A)}  - \abs{\log(1 + f_* \Dim A)}
\text{.}
\end{equation}
%\begin{exercise}
%    \begin{subequations}
%    Use \cref{eq:dim_ineq} for:
%    \begin{xca}
%        Prove that for a map $f$ from a $\mathrm{T}_4$-space $\situs$ to a metrizable space $\situs*$
%        \begin{align}\label{eq:lighness_snorm}
%            \Norm{f}{dim}   &=   \sup_{y \in Y} \abs*{\log (1 + \Dim (f^*\{y\}))}
%        ,\quad\text{and}
%        \\
%            \Dim \situs     &\leq\situs* + \exp \Norm{f}{dim} 
%        \text.
%        \end{align}
%    \end{xca}
%    \end{subequations}
%\end{exercise}
%\begin{exerciseSolved}

Another consequence of \cref{eq:dim_ineq} is that
\begin{equation}\label{eq:lighness_snorm}
        \Norm{f}{dim}
    =
        \sup_{y \in Y} \abs*{\log (1 + \Dim (f^*\{y\}))}
\end{equation}
for a map $f$ from a $\mathrm{T}_4$-space $\situs$ to a metrizable space $\situs*$:
Indeed, for any singleton $\{y\}$ we have $ \abs{\log(1+\Dim \{y\})} = 0 $ and hence $
            \norm{f}_{\Dim} 
    \geq    \sup_{y \in Y} \abs*{\log (1 + \Dim f^*y)} 
$. But also for any $C \subseteq Y$ we can apply Hurewicz's formula \cref{eq:dim_ineq} to $ f|_{f^* C} f^* C \to C $ getting $             \Dim (f^*C) 
    \leq    \Dim (C) + \sup\limits_{y\in C} \Dim f^*y
    \leq    \Dim (C) + \sup\limits_{y\in Y} \Dim f^*y
$.
Hence by sub-linearity and monotonicity of $\log$ we have $ 
    \log (1 + \Dim (f^*C)) \leq \log (1 + \Dim C) + \sup\limits_{y\in Y} \log(1 + \Dim f^*y )
$.
Since $C$ was arbitrary this proves the lemma.

Actually, \cref{eq:lighness_snorm} allows one to rephrase Hurewicz's formula \cref{eq:dim_ineq} as
\begin{equation}
    \Dim \situs \leq \situs* + \exp \Norm{f}{dim} 
\end{equation}
since surjectivity of $f$ implies that $\Dim (f^* \{y\}) \leq 0$.
%This version of the dimension norm actually goes back to \textcite{Hurewicz33}.
%Hurewicz 1933 "Über dimensionserhöhende stetige Abbildungen" explores what happens when a map raises dimension
\end{subequations}
%\end{exerciseSolved}

\subsection{Component seminorm}

Following \textcite{Whyburn50,CarboniJanelidzeKellyPare97} we call a map $f\colon \situs \to \situs*$ \definiendum{monotone} if the preimage of every singleton $\{y\}\subset \situs*$ is nonempty and connected. Note that this property implies surjectivity since the empty set consists of zero connected components.
For any topological space $\situs $ let $\operatorname{\mathcal{I}}(\situs) = (\operatorname{I}(\situs), \situsTop[\mathcal{I}(\situs)])$ denote the collection of connected components of $\situs$ endowed with the quotient topology.
Define the \definiendum{component seminorm} as
\begin{align*}
        \Norm{f}{comp}
    &\coloneqq
          \norm{f}_{|\log\#I|}
\\
    &=
        \sup\nolimits^0 
            \setBuilder*
                {\abs{\log(\# (\operatorname{I} f^*C))} - \log(\# (\operatorname{I} C))}
                {\substack{C \subseteq \situsSet* \text{ closed}, \\ \abs{\log \#(\operatorname{I} C)} < \infty }}
\\
    &=
        \sup\nolimits^0 
            \setBuilder*
                {\abs{\log(\# (\operatorname{I} f^*C))} - \log(\# (\operatorname{I} C))}
                {\substack{C \subseteq \situsSet* \text{ closed}, \\ 0 < \# (\operatorname{I} C) < \infty }}
\end{align*}
where we used the convention $\log(0) = - \infty$.
Note that we can express the number of connected components $\# I \situs$ of a nonempty space $\situs$ as $ 
        \exp \Norm{\situs \to \{*\}}{comp}
$ where $\situs \to \{*\}$ is the canonical map to the singleton space.
This norm relates to the dimension of fibers by the obvious inequality
\begin{equation}\label{eq:disconnNorm_fiber}
        \Norm{f}{comp}
    \geq
        \supBd_{p \in \situsSet* } 
            \abs{\log(\# (\operatorname{I} f^*\{p\})}
    \eqqcolon
        \operatorname{mon}(f)
\text.
\end{equation}

\begin{lemma}\label{lem:monotonNorm}
    Let $f\colon \situs \to \situs*$ be a \contFct.  Then
    \begin{claims}
        \item\label{claim:lem:monotonNorm_singleComponent}
            We have
            $
                    \Norm{f}{comp}
                =
                    \supBd
                        \setBuilder*
                            {\abs{\log (\# (\operatorname{I} (f^* C)))}}
                            {\substack{\emptyset \neq C \subseteq \situsSet* \text{ closed}, \\ \# \operatorname{I} C = 1}}
            \text.
            $
        \item\label{claim:lem:monotonNorm_mon}
            The map $f$ is monotone if and only if $\operatorname{mon}(f) = 0$.
        \item\label{claim:lem:monotonNorm_closedMonotone_impl_zeroNorm}
            If $f$ is closed and monotone then $\Norm{f}{comp} = 0$.
        \item\label{claim:lem:monotonNorm_compact}
            Assume that $\situs$ is compact and that $\situs*$ is Hausdorff.
            Then $f$ is monotone, if and only if $\Norm{f}{comp} = 0$.
    \end{claims}
\end{lemma}

\begingroup
\allowdisplaybreaks
\begin{exercise}
\begin{xca}
    Prove \cref{claim:lem:monotonNorm_singleComponent,claim:lem:monotonNorm_mon} of \cref{lem:monotonNorm}.
\end{xca}
\end{exercise}

\begin{proof}
\begin{exerciseSolved}
    For \cref{claim:lem:monotonNorm_singleComponent} observe
    \begin{align*}
            \Norm{f}{comp}
        &=
            \supBd_{\substack{C \subseteq \situsSet*, \\ 0 < \# (\operatorname{I} C) < \infty }} 
            \abs{\log(\# (\operatorname{I} f^*C))} - \log(\# (\operatorname{I} C))
    \\
        &\leq
            \supBd_{\substack{C \subseteq \situsSet*, \\ 0 < \# (\operatorname{I} C) < \infty }} 
            \abs*{\log\left(\sum_{ C' \in (\operatorname{I} C) }\# (\operatorname{I} f^*C')\right)} 
            - \log(\# (\operatorname{I} C))
    \\
        &=
            \supBd_{\substack{C \subseteq \situsSet*, \\ 0 < \# (\operatorname{I} C) < \infty }} 
            \begin{cases*}
                \infty                  
                &if $\exists C' \in \operatorname{I} C \colon \# (\operatorname{I} f^*C') = \infty $
            \\
                \log\frac{\sum_{ C' \in \operatorname{I} C }\# (\operatorname{I} f^*C')}{\# (\operatorname{I} C)}
                &else
            \end{cases*}
    \\
        &\leq
            \supBd_{\substack{C \subseteq \situsSet*, \\ 0 < \# (\operatorname{I} C) < \infty }} 
            \begin{cases*}
                \infty                  
                &if $\exists C' \in \operatorname{I} C \colon \# (\operatorname{I} f^*C') = \infty $
            \\
                \log\sup\limits_{ C' \in {\operatorname{I} C} } \# (\operatorname{I} f^*C')
                &else
            \end{cases*}
    \\
        &=
            \supBd_{\substack{ \emptyset \neq C \subseteq \situsSet*, \\ \# \operatorname{I} C = 1}} 
            \abs{\log (\# (\operatorname{I} C))}
    \\
        &=
            \supBd
                \setBuilder*
                    {\abs{\log (\# (\operatorname{I} C))} - \log(\# (\operatorname{I} C))}
                    {\substack{ \emptyset \neq C \subseteq \situsSet* \text{ closed}, \\ \# \operatorname{I} C = 1}}
    \\
        &\leq
            \supBd
                \setBuilder*
                    {\abs{\log(\# (\operatorname{I} f^*C))} - \log(\# (\operatorname{I} C))}
                    {\substack{ \emptyset \neq C \subseteq \situsSet* \text{ closed}, \\ 0 < \# (\operatorname{I} C) < \infty }}
    \\ \displaybreak[0]
        &=
            \Norm{f}{comp}
    \text.
    \end{align*}
    \Cref{claim:lem:monotonNorm_singleComponent} clearly follows.
    
    \Cref{claim:lem:monotonNorm_mon} follows from the fact that $\abs{\log(\# (\operatorname{I} f^* \{p\}))} = 0$ if and only if $ f^* \{p\} $ is a singleton.
\end{exerciseSolved}

\begin{exercise}
    \Cref{claim:lem:monotonNorm_singleComponent,claim:lem:monotonNorm_mon} are already proved by the exercise.
\end{exercise}
    For \cref{claim:lem:monotonNorm_closedMonotone_impl_zeroNorm} assume that $f$ is monotone and closed.
    Note that the restriction $f|_{f^*C}\colon f^*C \to C$ is closed as well: 
    take any relatively closed $ A \subseteq f^* C $. Any point $ x $ in the closure of $f_*(A)$ relative to $C$ must be in the image $f_*(\overline{A}^{\situs})$ of the closure of $C$ in $\situsSet$, but then already $y \in A$ for any $y\in f^* x$. 
    Thus $f_*(A) = f_*(\overline{A}^{\situs})$.
    
    Let $C$ be an arbitrary closed connected subset of $\situsSet*$. 
    The preimage $f^*C$ must not be empty because otherwise monotonicity of $f$ would be violated.
    Since $C$ is closed, so is $f^*C$.  Assume that $f^*C$ is a disjoint union of two sets $K$ and $L$ that are clopen in the relative topology on $f^*C$.
    For any $y \in C$ the preimage $f^*\{y\}$ is connected in the relative topology.
    Hence either $f^*\{y\} \subseteq K$ or  $f^*\{y\} \subseteq L$.
    Since this holds for all $y \in C$, the set $C$ is actually the disjoint union of $f_*K$ and $f_*L$.
    Due to our observation on the closedness of $f|_{f^*C}\colon f^*C \to C$ both $f_*K$ and $f_*L$ are open in the relative topology.
    Thus by connectedness of $C$ either $f_*K = \emptyset$ or $f_*L = \emptyset$, a contradiction.
    Consequently, $f^*C$ is connected.
    Hence $\Norm{f}{comp} = 0$.

    In \cref{claim:lem:monotonNorm_compact} the direction $ \text{monotone} ``\ldots \implies \Norm{f}{comp} = 0 $'' follows from \cref{claim:lem:monotonNorm_closedMonotone_impl_zeroNorm} and the fact that a \contFct{} from a compact space to a Hausdorff space is closed. The other direction is implied by \cref{eq:disconnNorm_fiber,claim:lem:monotonNorm_mon}.
\end{proof}
\endgroup

%\begin{subenvironments}
%\begin{remark}
%    If in  we had defined $\Norm{}{comp}$ by $C$ ranging over all connected subsets of $\situs*$, then \cref{claim:lem:monotonNorm_closedMonotone_impl_zeroNorm} still holds, meaning that if a map $f$ is closed and monotone, then every preimage of any connected set is connected.
%\end{remark}

\begin{exerciseSolved}

\begin{remark}
    In \cref{claim:lem:monotonNorm_compact} the assumption that $\situsTop*$ is Hausdorff
    is necessary for the the statement that monotonicity of $f$ implies $\Norm{f}{comp} = 0$:
    For a counterexample let $\situs$ be the discrete space on a two element set and $\situs*$ the Sierpi\'nski space, i.e.\ $\situsSet* = \{0,1\}$ and $\situsTop* = \{ \emptyset, \{1\}, \situsSet* \}$. 
    Being finite $\situs$ is compact.
    Let $ f\colon \situs \to \situs* $ be a bijection. Since the map $f$ is bijective, it is monotone. 
    On the other hand $\situs*$ is connected but its preimage consists of two connected components.
\end{remark}
\end{exerciseSolved}

\begin{exercise}
\begin{xca}
    Find an example of a monotone map $f$ from a compact space into a non-Hausdorff space such that $\Norm{f}{comp} > 0$.
\end{xca}
\end{exercise}
%\end{subenvironments}

\begin{theorem}
    A map $f\colon \situs\to\situs*$ between totally disconnected compact Hausdorff spaces having $\Norm{f}{comp}=0$ is a homeomorphism.
\end{theorem}

\begin{proof}
    Assume that $\Norm{f}{comp}=0$.
    Then this map is surjective.
    Since every fiber is totally disconnected, $\Norm{f}{comp}=0$ implies that each fiber is a singleton.
    Hence $f$ is bijective.
    As for compact Hausdorff spaces the notions of closed and compact subsets coincide, and compact subsets are mapped to compact subsets under \contFct{}, the inverse of $f$ is continuous.
    Thus $f$ is a homeomorphism.
\end{proof}

\subsection{The topological norm}

Define the \definiendum{topological norm} as
\[
    \Norm{f}{top} \coloneqq \Norm{f}{comp} + \Norm{f}{dim}
\text.
\]

\begin{proposition}\label{prop:normTop0_impl_mon_light}
    Let $\situs$ be a compact $\mathrm{T}_4$ space and $\situs*$ be metrizable.
    If $\Norm{f}{top}=0$ for a \contFct{} $\situs\to\situs*$, then $f$ is monotone and light.
\end{proposition}

\begin{proof}
    Assume that $\Norm{f}{top}=0$.
    Then $\Norm{f}{comp} = \Norm{f}{dim} = 0$.
    The fact $\Norm{f}{comp} = 0$ implies that $f$ is monotone because points in $\situs*$ are closed.
    The other fact $\Norm{f}{dim} = 0$ implies that $f$ is light by \cref{claim:lem:monotonNorm_compact} of \cref{lem:monotonNorm}.
\end{proof}

\begin{theorem}
    Let $\situs$ be a compact $\mathrm{T}_4$ space and $\situs*$ be metrizable.
    If there is map  with $\Norm{f}{top}=0$, then $f$ is a homeomorphism.
    Especially, the category of compact metrizable spaces is a normed category with respect to $\Norm{}{top}$.
\end{theorem}

\begin{proof}
    By \cref{prop:normTop0_impl_mon_light} the map $f$ is monotone and light.
    Since the identity on a topological space is monotone and light as well, we have two factorizations of $f$ in a monotone and a light map:
    
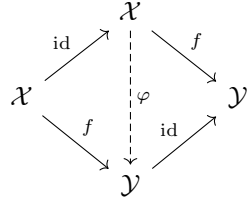
\begin{equation*}
        \begin{tikzcd}
                                            & \situs\ar[dr, "f"]    &           \\
            \situs\ar[ur, "\id"]\ar[dr, "f"]&                       & \situs*   \\
                                            & \situs*\ar[ur, "\id"] &
            \ar[from=1-2, to=3-2, dashed, "\varphi "]
        \end{tikzcd}
    \end{equation*}
    By the classical uniqueness of the monotone-light factorization \cite[\nopp 2.8, 7.3]{CarboniJanelidzeKellyPare97} there is a homeomorphism $\varphi\colon \situs\to\situs*$.
\end{proof}

\section{Metric spaces}\label{sec:Met}
Let $\Cat{Met}$ denote the category of compact metric spaces with multi-valued functions among them as morphisms, 
    i.e.\ functions $f\colon \spSet \to \powerSet(\spSet*)\setminus\{\emptyset\}$, $ x \mapsto f[x] $.
We will always write $f[x]$ instead of $f(x)$ for multi-valued functions to avoid any confusion with normal functions.
Objects of $\Cat{Met}$ we denote by curly letters, e.g.\ $\spSp = (\spSet, \spMet)$, and the metric is abbreviated by $\dist{\blank}{\blank} = \spMet(\blank, \blank)$ if no confusion can arise.

%By the notation $\varphi(f(x))$ for some formula $\varphi$ we mean $\exists y \in %f[x]\colon \varphi(y)$, e.g.\ $y = f(x)$ means $y \in f[x]$.
%Set builder notation has to be understood as $ \setBuilder{f(x)}{x \in M} = \setBuilder{y}%{y = f(x), x \in M} $.
%This also explains how suprema and infima are to be understood.
%The image and preimage functions as defined as $ f_*(A) = \bigcup_{x\in A} f(x) $ and $ f^*(A) = \setBuilder{x}{f(x) \cap A \neq \emptyset} $, respectively.

Actually, our arguments in this section extend to densely defined multi-valued functions, i.e.\ functions $ f\colon  \spSet \to \powerSet(\spSet*)$ such that points with $f[x]\neq \emptyset$ are dense. This is done by transforming such a function to a morphism in $\Cat{Met}$ by the \definiendum{closure} 
\begin{equation}\label{eq:denslyDefFct_closure}
        \overline{f}[x]
    \coloneqq
        \setBuilder*{y}{ y = \lim_{n\to\infty} y_n \text{ for } y_n \in f[x_n] \text{ with } x_n \xrightarrow{n\to \infty} x }
    \text.
\end{equation}

Let's start with a naïve approach by considering subsets $A \subseteq \spSet$ as subobjects and considering the perhaps most natural quantity as a capacity, the \definiendum{diameter}
\[
    \diameter = \lambda A.\supBd_{x,y \in A} \dist{x}{y}.
\]
This induces the seminorm
\[
        \norm{f}_{\diameter}'
    =   
        \supBd_{A \subseteq \spSet*} \left[\diameter(f^* A) - \diameter(A)\right]
\]
for a morphism $f\colon \spMetDef \to \spMetDef*$.
Unfortunately, as we want to include multi-valued functions to our discussion, the adjunction $f_* \dashv f^*$ is no longer valid for the power sets.

But there is a solution to this problem: The idea is to choose the power set of a power set as set of subobjects. Moreover,---since a metric is a map on pairs of points---it's suitable to consider the Cartesian product $\spSet \times \spSet $ in lieu of $\spSet$.
This is done by making $\Cat{Met}$ into a concrete category with generalized subobject (cf. \cref{def:concreteCat_subObj}) as follows:
\begin{center}
    \begin{tikzcd}
        \Cat{SMet} \ar[dr, "F_S"]     &        \\
        \Cat{Met}\ar[u, "S"]\ar[r, "F"]  &   \Cat{Set}
    \end{tikzcd}
\text,
\end{center}
\begin{subequations}
\begin{align}
    F(\spSp)    &=  \powerSet(M\times M)
\\
    F(f)        &= \lambda P. \setBuilder{(y, y')}{y \in f[x], y'\in f[x'], (x,x') \in P},
\\
    \Cat{SMet}    &= \Cat{Set}
\\
    S           &=  F
\\
    \operatorname{GS}(\spSp) &= \powerSet(F(\spSp))
\end{align}
for $\spSp\in\Cat{Met}_0$ and $f\in\Cat{Met}_1$.
\end{subequations}
Note that by this definition subobjects are simply normal subsets (i.e.\ subobjects in set) of the power set. Hence the pullback of subobjects is just the preimage, i.e.
\[
    (Sf)^*(B) = \setBuilder{P\subseteq \powerSet( M\times M )}{(Ff)(P)\in B}
\text.
\]

On the subobjects of a $S(\spSp)$ for $\spSp=(\spSet, \spMet)\in \Cat{Met}_0$ we define the capacity
\begin{equation}\label{eq:snorm_diam_def}
        c_{\textnormal{diam}}
    =
        \lambda A. \adjustlimits{\supBd}_{P\in A}{\inf}_{p\in P} \spMet(p)
\end{equation}
which is actually a capacity, i.e. monotone, being defined by a supremum.
Any subset $A \subseteq \spSet$ naturally corresponds to $ \powerSet(A\times A) $. Observe
\begin{equation}
    \diameter(A) = c_{\textnormal{diam}}(\powerSet(A\times A))
\text.
\end{equation}

\begin{subequations}
By definition and the fact that for any compact space there is a maximum value  (i. e. the diameter of the space) attained for the distances of two point sets we have
\begin{equation}
        \Norm{f}{diam}
    =
        \supBd\setBuilder
            {c_{\textnormal{diam}}( (S f)^*(B) ) - c_{\textnormal{diam}} ( B )}
            { B \subseteq \powerSet(\spSet* \times \spSet*) }
\text.
\end{equation}
As with $\Cat{SMet} = \Cat{Set}$ we have the usual adjunction---i.e.\ Galois connection in this case---$f_*\dashv f^*$. Thus we can apply \cref{lem:capacities_Galois} and again finiteness of the diamter obtaining
\begin{align}
\label{eq:snorm_diam_sup}
        \Norm{f}{diam} 
    &= 
        \supBd\setBuilder
            {c_{\textnormal{diam}}(A) - c_{\textnormal{diam}} ( (S f)_*(A) ) }
            {A \subseteq \powerSet(\spSet\times\spSet)}
\\ \nonumber
    &=  \supBd\setBuilder
            {c_{\textnormal{diam}}(\{P\}) - c_{\textnormal{diam}} ( (S f)_*(\{P\}) ) }
            {\{P\} \subseteq \powerSet(\spSet\times\spSet) }
\intertext{where "$\geq$" is obvious and "$\leq$" holds since 
    $ c_{\textnormal{diam}}(A) = \sup_{P \in A} c_{\textnormal{diam}}(\{P\}) $. }
\nonumber
    &= \supBd\setBuilder*
            {\inf_{p\in P} \spMet(p) - \inf_{p\in (S f) ( P ) } \spMet(p) }
            {P \subseteq \powerSet(\spSet\times\spSet) }
\\ \nonumber
    &= \supBd\setBuilder*
            {\inf_{p\in P} \spMet(p) - \spMet(q) }
            {P \subseteq \powerSet(\spSet\times\spSet), q\in (S f) ( P ) }
\\ \nonumber
    &= \supBd\setBuilder*
            {\spMet( p) - \spMet(q)  }
            { p \in \powerSet(\spSet\times\spSet), q\in (S f) ( P ) }
\\
\label{eq:snorm_diam_ptWise_sup}
    &=  \supBd \setBuilder{\spMet(x, x') - \spMet*(y, y')}{x, x' \in M, y\in f[x], y'\in f[x']}.
\end{align}
\end{subequations}

For $r>0$ let $\spSp_r = (\{x, y\}, \Dist_r)$ be the two point metric space such that $\Dist_r(x,y) = r$.
\begin{exercise}
This definition may serve as a hint for the following exercises.
\begin{xca}\label{lem:dil_duals}
    Show that the seminorm $\Norm{}{diam}$ is left reflexive and has the left dual
    \begin{align}\label{eq:norm_dil_leftDual}
            \Norm{f}{diam}^{*\mathrm{L}}
        &=
            \supBd_{x,y \in \spSet } \left(\dist{f(x)}{f(y)} - \dist{x}{y}\right)
    \text.
    \end{align}
\end{xca}
\end{exercise}
\begin{exerciseSolved}
Next we calculate
\begin{lemma}\label{lem:dil_duals}
    The seminorm $\Norm{}{diam}$ has the left dual
    \begin{equation}\label{eq:norm_dil_leftDual}
            \Norm{f}{diam}^{*\mathrm{L}}
        =   
            \supBd \setBuilder*{ \dist{y}{y'} - \dist{x}{x'} }{x,x' \in \spSet, y\in f[x], y' \in f[x'] }
    \end{equation}
    and is left reflexive.
\end{lemma}

\begin{proof}
    The claim \cref{eq:norm_dil_leftDual} follows from the estimate
    \begin{align*}
            \Norm{f}{diam}^{*\mathrm{L}}
        &=  \sup_{f'} \left(\Norm{f'}{diam} - \Norm{f' \comp* f}{diam} \right)
    \\
        &\stackrel{\cref{eq:snorm_diam_ptWise_sup}}{=}  
            \sup_{f'} \Big(
                    \supBd_{\substack{z, z' \in \spSet',\\ x\in f'[z], x'\in f'[z']} } (\dist{z}{z'} - \dist{x}{x'} )
                -   \supBd_{\substack{z, z' \in \spSet',\\ y\in ff'[z], y'\in ff'[z']} } (\dist{x}{y} - \dist{y}{y'})
            \Big)
    \\  
        &\leq  
            \sup_{f'} \supBd_{z, z' \in \spSet', x\in f'[z], x'\in f'[z'], y\in f[x], y'\in f[x'] } \bigl(
                (\dist{z}{z'} - \dist{x}{x'} ) - (\dist{z}{z'} - \dist{y}{y'})
                \bigl)
    \\
        &\leq  
            \sup_{f'} \supBd_{z, z' \in \spSet', x\in f'[z], x'\in f'[z'], y\in f[x], y'\in f[x'] } \bigl(
                    \dist{y}{y'} - \dist{x}{x'}
                \bigl)
    \\
        &\leq
            \supBd_{x,x' \in \spSet, y\in f[x], y' \in f[x'] } \left(\dist{y}{y'} - \dist{x}{x'}\right)
    \\
        &=
            \supBd 
                \setBuilder*
                    {( r - \dist{y}{y'}) - ( r - \dist{x}{x'} )}
                    { \substack{ r>0, x, x' \in M \text{ with } \dist{x}{x'} = r,\\ y\in f[x], y' \in f[x']} }
    \\
        &=
            \supBd
                \setBuilder*
                    {\Norm{f'}{diam} - \Norm{f' \comp* f}{diam}}
                    {f'\colon \spSp_r \to \spSp \text{ with } r>0}
    \\
        & \leq
            \sup_{f'} \left(\Norm{f'}{diam} - \Norm{f' \comp* f}{diam} \right)
    \text.
    \end{align*}
    
    Left reflexivity follows from \cref{eq:rightBidual_le_norm}, i.e.\ $
        \Norm{}{diam}^{*\mathrm{L}*\mathrm{L}} \leq \norm{} 
    $, and the estimate
    \begin{align*}
            \Norm{}{diam}^{*\mathrm{L}*\mathrm{L}}
        &=
            \sup_{f'} \left(\Norm{f'}{diam}^{*\mathrm{L}} - \Norm{f' \comp* f}{diam}^{*\mathrm{L}} \right)
    \\
        &\geq
            \supBd_{f'\colon \spSp_r \to \spSp } 
                \left(\Norm{f'}{diam}^{*\mathrm{L}} - \Norm{f' \comp* f}{diam}^{*\mathrm{L}} \right)
    \\
        &\stackrel{\cref{eq:norm_dil_leftDual}}{=}
            \supBd
                \setBuilder*
                    {(\dist{x}{x'} - r) - (\dist{y}{y'} - r)}
                    {\substack{r>0, x, x' \in M \\ \text{with } \dist{x}{x'} = r,\\ y\in f[x], y' \in f[x'] }}
    \\
        &=
            \supBd_{x,y\in\spSet } \dist{x}{y} - \dist{f (x)}{f (y)}
    \\
        &\stackrel{\cref{eq:snorm_diam_ptWise_sup}}{=}
            \Norm{}{diam}
    \text.
    \qedhere
    \end{align*}
\end{proof}
\end{exerciseSolved}

\begin{exercise}
\begin{subenvironments}
\begin{xca}
    Show that $
        \Norm{f}{diam}^{*\mathrm{L}} \geq  \norm{f}_{-\diameter}
    $.
    Hint: Use \cref{lem:dil_duals}.
\end{xca}
\begin{xca}
    Find a counterexample to the reverse inequality $\Norm{f}{diam}^{*\mathrm{L}} \leq  \norm{f}_{-\diameter}$.
\end{xca}
\end{subenvironments}
\end{exercise}
\begin{exerciseSolved}
\begin{lemma}
    We have that $
        \Norm{f}{diam}^{*\mathrm{L}} \geq  \norm{f}_{-\diameter}
    $.
\end{lemma}
\begin{proof}
    Expressing $\Norm{f}{diam}^{*\mathrm{L}}$ by \cref{lem:dil_duals} we get the desired estimate:
    \begin{align*}
            \Norm{f}{diam}^{*\mathrm{L}}
        &=   
            \supBd_{x,y \in \spSet } 
                \left(\dist{f(x)}{f(y)} - \dist{x}{y}\right)
    \\ &=    \supBd_{x,y \in \spSet } 
                - \diameter\{x, y\}
                - (-\diameter\{f(x), f(y)\} )
    \\  &\geq\supBd_{x,y \in \spSet }
            - \diameter (f^*\{f(x), f(y)\})
            - \diameter\{x, y\}
                - (-\diameter\{f(x), f(y)\} )    
    \\  &\geq
            \supBd_{A \subset \spSet* } 
                - \diameter(f^*A)
                -(-\diameter A)
    \qedhere
    \end{align*}
\end{proof}
\begin{example}
    The reverse inequality $\Norm{f}{diam}^{*\mathrm{L}} \leq  \norm{f}_{-\diameter}$ does not hold as is seen from the example of the space $ \{0,1,2\} \subset \R $ with the induced distance and the map $ 0 \mapsto 0; 1,2 \mapsto 2 $ to $\R$: we have $ \Norm{f}{diam}^{*\mathrm{L}} \geq \dist{0}{2} - \dist{0}{1} = 1 $, but the preimage of any set containing $0$ and $1$ has always diameter 2.
\end{example}
\end{exerciseSolved}

Note that a function $f$ with $\Norm{f}{diam}^{*\mathrm L}$ bounded is single-valued.
Moreover let $T$ be the one point metric space. Observe that
\begin{equation}
    \diameter M = \Norm{M\to T}{diam}
\text{.}
\end{equation}
Further set $
        \Dist[dil](\spSp,\spSp*)
    \coloneqq
        \Dist_{\Norm{}{diam}}(\spSp,\spSp*)
    =
        \inf \setBuilder{\Norm{f}{diam}}{ f\colon \spSp \to \spSp* }
$.
We recall the well-known notion of Gromov-Hausdorff distance \cites[11.1.1]{Petersen16}[\S~3A]{Gromov99} employing the following shorthand notations for any $A \subseteq \spSet$
\begin{align*}
    A^{r)} &\coloneqq \setBuilder{x \in \spSet}{ \dist{x}{A} < r }
&&\text{for $ r > 0 $ and}\quad
\\
    A^{r]} &\coloneqq \bigcap_{r' > r} A^{r')}
&&\text{for $r \geq 0$.}
\end{align*}
A subset $X \subseteq \spSet $ is said to be \definiendum{$l$-dense in $\spSet$} if $ X^{l]} = \spSet $.
Let $A, B \subseteq \spSet$ be subsets of a metric space $\spSp$.
The \definiendum{Hausdorff distance} between $A$ and $B$ is given by
\begin{align*}
        \Dist[H](A, B)
    &\coloneqq
        \inf \setBuilder
            {r\in [0,\infty] }
            { A \subseteq B^{r]} \text{ and } B \subseteq A^{r]} }
\text{.}
\intertext{Let $M$, $N$ be metric spaces. Their \definiendum{Gromov-Hausdorff distance} is}
        \Dist[GH](\spSp, \spSp*)
    &\coloneqq
        \inf\setBuilder
            {\Dist[H](f_* \spSet, g_* \spSet*)}
            {
                \spSp \xrightarrow{f} \spSp[L]
                \xleftarrow{g} \spSp*
            }
\end{align*}
where $\spSp[L]$ ranges over all metric spaces and $f, g$ are metric embeddings.
Recall that a function is \definiendum{Cauchy continuous} if it preserves Cauchy sequences.

\begin{theorem}\label{thm:GHdist_locLipEquiv}
    The identity map on $ 
            \skeleton_0(\Cat{Met}, \Norm{}{diam})$
    with the Gro\-mov-Haus\-dorff metric $ \Dist[GH] $ on the domain and the distance $\Dist[diam]^+  
    $ on the codomain is 2-Lip\-schitz with Cauchy continuous inverse.
\end{theorem}

The next lemma is a quantitative version of \cite{FreudenthalHurewicz36}, cf.\ also \cite{BucicovschiMeyer15}\footnote{Note that the result of the latter is implied by the former using the closure \cref{eq:denslyDefFct_closure}.}
For its proof we require the terminology of packings, which can be elegently be introduced by what we just developed: 

For any metric space $\spSp\in \Cat{Met}$ a collection $ P \coloneqq \{p_1, \ldots, p_n \}\subseteq \spSet$ is called a \definiendum{packing of $\spSp$}. Further let the \definiendum{configurations} of a packing be $\configurations P \coloneqq \setBuilder{(p, q)\in X^{\times 2}}{p \neq q }$, the set of all pairs $(x,x')$ of distinct points $x$ and $x'$ from $X$. 
Further, we can assign a map $\operatorname{pack}_{\spSp}\colon \N \to \operatorname{GS}(\spSp)$
\[
        \operatorname{pack}_{\spSp}
    = 
        \lambda n. \setBuilder{ \configurations(P) }{P \subseteq \spSet, \# P = n}
\text.
\]
We can compose this map with the capacity $c_{\textnormal{diam}}$
\begin{equation}
    \N \xrightarrow{\operatorname{pack}_{\spSp} }
    \operatorname{GS}(\spSp) \xrightarrow{c_{\textnormal{diam}}}
    [0,\infty]
\end{equation}

For $l > 0$ an \definiendum{$l$-packing of $\spSp$} is a packing $P$ that $ c_{\textnormal{diam}} \configurations(P) > l $.
We define the \definiendum{metric $l$-packing number of $\spSp$} by
\begin{align}
    \card[pack]_l (\spSp)
    &\coloneqq  \sup \operatorname{pack}_{\spSp}^*(c_{\textnormal{diam}}^* (\left(l,\infty\right]) )
\\
    &=  \supBd\setBuilder{n}{\exists l\text{-packing } (p_1, \ldots, p_n) \text{ of } \spSp }
\text.
\end{align}
This definition is extended to non-positive $l$ by $\card[pack]_l (\spSp) = \infty$ (also for the terminal space).
As $\spSp$ is compact, $N \coloneqq \card[pack]_l (\spSp)$ is finite. A collection $ P \coloneqq (p_1, \ldots, p_N) $ is called a \definiendum{maximal $l$-packing of $\spSp$}. 

Further define the \definiendum{total distances} for a finite $P\subseteq \spSet$ and of $\spSp$ itself by
\begin{align*}
        \Norm{P}{}[tot]
    &\coloneqq
        \sum_{(p,q) \in \configurations P} \dist{p}{q}
&&\text{and}
\\
          \Norm{\spSp}{}[tot]_l
    &\coloneqq
        \sup\setBuilder
            {\Norm{P}{}[tot] }
            { P \text{ is an $l$-packing of } \spSp }  
\text{.}
\end{align*}
The integer valued function
\[
    l \mapsto \card[pack]_l (\spSp)
\]
is continuous from the left and monotonically decreasing in $l$.
Finally, for spaces $\spSp$ and $\spSp*$ with it is easy to see that 
\begin{equation}\label{eq:packDist_ineq}
    \card[pack]_l (\spSp*)
    \leq
    \card[pack]_{ l - \Dist[diam](\spSp*, \spSp) } (\spSp),
\end{equation}
indeed, given an $l$-packing $P \subseteq \spSet* $ with $\Norm{P}{}[tot] = \Norm{\spSp*}{}[tot]_l$
the set $f_*(P)$ is still an $l - \Dist[diam](\spSp*, \spSp)$-packing, 
provided that $l - \Dist[diam](\spSp*, \spSp) > 0$, or the right hand side if $\infty$, otherwise.

\begin{lemma}\label{lem:almostExpansiveEndo}
    Let $\spSp, \spSp'$ be compact metric spaces.
    For all $L,l$ with $ l > L \geq 0 $ it holds for sufficiently small $\delta > 0 $ that for every $h\colon \spSp' \to \spSp'$ with $ \Norm{h}{diam} < \delta $ and $ \Dist[diam](\spSp', \spSp) \leq L $ we have that
    \begin{claims}
        \item\label{claim:lem:almostExpansiveEndo_dense} 
            $h_*(\spSet')$ is $l$-dense, and
        \item\label{claim:lem:almostExpansiveEndo_upperDil} 
            $\Norm{h}{diam}^{*\mathrm{L}} \leq 4l + C\delta$ where $C=C( l- L, \spSp)$.
    \end{claims}
\end{lemma}

\begin{proof}
    For the first claim, let $h\colon \spSp' \to \spSp'$ be a map between compact metric spaces.
    By monotonicity and continuity from the left of  $l \mapsto \card[pack]_l (\spSp)$
    we can find some small $\varepsilon > 0$ such that for all $\delta \in (0,\varepsilon]$, we have $
        \card[pack]_{l-\delta} (\spSp') = \card[pack]_l (\spSp') 
    $.
    For any $l$-packing $P = (p_1, \ldots, p_N)$ with $N \coloneqq \card[pack]_l (\spSp')$ and $h$ with $\Norm{h}{diam} < \delta \leq \varepsilon$ the collection $ h_*(P) $ is an $(l-\delta)$-packing.
    Since $\card[pack]_{l - \delta} (\spSp') = \card[pack]_l (\spSp') = \card[pack]_{l - \varepsilon} (\spSp')$ this implies that $ h_*(P) $ is even a maximal $(l-\delta)$-packing by \cref{eq:packDist_ineq}.
    Hence $h_*(\spSet')$ is $l$-dense; actually even $(h_*(P))^{l)} = \spSet'$. Thus \cref{claim:lem:almostExpansiveEndo_dense} holds.
    
    For \cref{claim:lem:almostExpansiveEndo_upperDil}, i.e.\ $\Norm{h}{diam}^{*\mathrm L} \leq 4l + C\delta$, assume further that $ \Dist[diam](\spSp', \spSp) \leq L $.
    Observe that it is possible to find an $l$-packing $P$ in $ h^*\spSet' $ such that
    \[
        \Norm{P}{}[tot] > \Norm{\spSp'}{}[tot]_l - \big(\card[pack]_l (\spSp')\big)^2 \delta
    \text{.}
    \]
    Thus by \cref{eq:packDist_ineq}
    \[
        \Norm{P}{}[tot] > \Norm{\spSp'}{}[tot]_l - \big(\card[pack]_{l-L} (\spSp)\big)^2 \delta
    \text{.}
    \]
    We still assume $\delta \leq \varepsilon$.
    Further observe that $
            \Norm{h(P)}{}[tot]
        \leq
            \Norm{\spSp'}{}[tot]_{l-\delta}
        =
            \Norm{\spSp'}{}[tot]_l
        <
            \Norm{P}{}[tot] + \big(\card[pack]_l (\spSp')\big)^2 \delta
        \leq
            \Norm{P}{}[tot] + \big(\card[pack]_{l-L} (\spSp)\big)^2 \delta
    $ and hence, 
    \begin{equation*}
            \sum_{(p,q) \in \configurations h(P)} \dist{p}{q}
        \leq
                \sum_{(p,q) \in \configurations P} \dist{p}{q}
            +   \big(\card[pack]_{l-L} (\spSp)\big)^2 \delta
    \text{.}
    \end{equation*}
    From $\dist{h(p)}{h(q)} \geq \dist{p}{q} - \delta $ and a summand-wise comparison we get that 
    for all $p,q \in P$ (using the notation $\tilde p = h(p), \tilde q = h(q)$),
    \begin{align}
    \nonumber
            \dist{\tilde p}{\tilde q} 
        &\leq 
                \dist{p}{q} 
            +   \big(\card[pack]_l (\spSp') - 1\big)^2\delta
            +   \big(\card[pack]_{l-L} (\spSp)\big)^2 \delta
    \\
    \nonumber
        &\leq
                \dist{p}{q} 
            +   \big(\card[pack]_{l-L} (\spSp) - 1\big)^2\delta
            +   \big(\card[pack]_{l-L} (\spSp)\big)^2 \delta
    \\
    \label{eq:packing_estimate}
        &\leq    \dist{p}{q} + C'\delta
    \end{align}
    where the parameter $C'$ depends upon $\spSp$ and $l - L$.
    
    To conclude the argument for \cref{claim:lem:almostExpansiveEndo_upperDil}, let $x,y \in \spSet'$.
    Set $ \tilde x = h(x) $ and $\tilde y = h(y)$.
    We derive an  estimate for $\tilde p = h(p)$ and $ \tilde q = h(p) $ with $ p,q \in  P$ such that $\dist{\tilde x}{\tilde p}, \dist{\tilde y}{\tilde q} < l$.
    Observe that $ \dist{x}{p}, \dist{y}{q} \leq l + \delta $, so we have
    \[
                \dist p q 
        \leq    \dist p x  +  \dist x y  +  \dist y q
        \leq    \dist x y  +  2 (l+\delta)
    \text.
    \]
    Now we can apply \cref{eq:packing_estimate}
    \[
                \dist {x} {y}
        \geq    \dist p q  -  2 (l+\delta)
        \geq    \dist {\tilde p} {\tilde q}  -  C'\delta  -  2 (l+\delta)
    \text.
    \]
    Finally, we obtain by setting $C \coloneqq C' + 2$
    \[
                \dist {\tilde x} {\tilde y}
        \leq    \dist {\tilde x} {\tilde p} + \dist {\tilde p} {\tilde q} + \dist {\tilde q} {\tilde y}
        \leq    l + \big(\dist x y + C' \delta + 2(l + \delta)\big) + l
        =       4l + C\delta
    \text.
    \qedhere
    \]
\end{proof}

This theorem in particular implies that in the category $\boxed{M}(\Cat{Met}, \Norm{}{diam}) $ every endomorphism is an isomorphism. Such categories are called \definiendum{EI-cate\-go\-ries} and have been studied for several decades \cite{tomDieck87}.

\begin{exercise}
\begin{xca}
    Find a counterexample showing that \cref{claim:lem:almostExpansiveEndo_dense} in \cref{lem:almostExpansiveEndo} does not hold when $\Norm{}{diam}$ is replaced by $ \norm{}_{-\diameter} $.
\end{xca}
\end{exercise}

\begin{exerciseSolved}
\begin{example}[counterexample]
    With regard to a global estimate on the density of $h(\spSet)$ in \cref{claim:lem:almostExpansiveEndo_dense} in \cref{lem:almostExpansiveEndo} we give a counterexample that shows that the claim does not hold for $ \norm{}_{-\diameter} $. 
    Consider the metric on $\spSet \coloneqq \{0,1,\ldots,n\}$ determined by
    \[
            \dist{i}{j}
        \coloneqq
            \begin{cases*}
                j - 1   & if $j\geq 2$  \\
                1       & if $j=1$
            \end{cases*}
    \]
    for $i<j$ and the map $h\colon \spSp \to \spSp$ defined by
    \[
            h(i) \coloneqq 0 \vee (i-1)
    \text{.}
    \]
    Indeed $\norm{h}_{-\diameter} = 0 $ but $n \notin h(\spSet)$ and, hence, $\Openball{n}{n-1} \subseteq (h_* \spSet)^{\mathrm{C}}$.
\end{example}
\end{exerciseSolved}

\begin{proof}[Proof of \cref{thm:GHdist_locLipEquiv}]
    Set $ \skeleton_0 \coloneqq \skeleton_0(\Cat{Met}, \Norm{}{diam})$.
    First, we prove the 2-Lip\-schitz property of $ \id \colon (\skeleton_0, \Dist[GH]) \to (\skeleton_0, \Dist[diam]^+) $.
    Set $l \coloneqq \Dist[GH](\spSp,\linebreak[0] \spSp*)$.
    For every $\varepsilon > 0$ we have embeddings $\spSp \xrightarrow{f} \spSp[L] \xleftarrow{g} \spSp*$ such that $f_* \spSet \subseteq (g_* \spSet*)^{l+\varepsilon]} $ and $ g_* \spSet* \subseteq (f_* \spSet)^{l+\varepsilon]}$.
    Set $h[x] \coloneqq \Closedball{x}{l+\varepsilon} \cap \spSet*$, where the ball and the intersection are in $\spSp[L]$.
    Observe that $
            \Norm{h}{diam}
        =   \sup_{x,y} \dist{x}{y} - \dist{f(x)}{f(y)}
        \geq\sup_{x,y} \dist{x}{f(x)} + \dist{y}{f(y)}
        \geq 2(l + \varepsilon)
    $.
    Since $\varepsilon > 0$ can be chosen arbitrarily small, $ 
        \Dist[diam](\spSp, \spSp*) \leq 2l 
    $. The analog argument with $\spSp$ and $\spSp*$ interchanged gives $ 
        \Dist[diam](\spSp, \spSp*) \leq 2l 
    $. This implies $ 
        \Dist[diam]^+(\spSp, \spSp*) \leq 2l 
    $.
    
    To show that $ 
        \id\colon (\skeleton_0, \Dist[diam]^+) \to (\skeleton_0, \Dist[GH])  
    $ is Cauchy continuous it suffices to show that for any Cauchy sequence $ \spSp_n $ with respect to $\Dist[diam]^+$ the following holds: 
    for all $N \in \N, L> 0$ we have that if $\forall n > N\colon \Dist[diam]^+(\spSp_N, \spSp_n) < \nicefrac L 2$,
    then $ \exists M \geq N\colon \forall n,m \geq M\colon \Dist[GH](\spSp_n, \spSp_m) \leq 5 L $.
    
    Take such $N$ and $L> 0$ so that for all $n > N$ we have $ \Dist[diam]^+(\spSp_N, \spSp_n) < \nicefrac L 2 $.
    We know that $ \Dist[diam](\spSp_n, \spSp_N) < L $ for all $n \geq N$.
    Let $C = C(\spSp_N, L)$ be the parameter from \cref{lem:almostExpansiveEndo}.
    Choose $M \geq N $ so large that for all $n,m > M$ there are maps $
        \spSp_n \xrightarrow{f_{nm}} \spSp_m \xrightarrow{g_{nm}} \spSp_m
    $ such that $ \Norm{f_{nm}\comp*g_{nm}}{diam} < \nicefrac L C \wedge L $ and $ \Norm{f_{nm}}{diam}, \Norm{g_{nm}}{diam} < L $.
    Set $ h_{nm} \coloneqq f_{nm}\comp*g_{nm}$.
    Hence by \cref{lem:almostExpansiveEndo} for sufficiently large $n$ 
        we have that $h_{nm !}(\spSet_n)$ is $L$-dense in $\spSp_n$ and $ \Norm{h_{nm}}{diam}^{*\mathrm{L}} \leq 5L $.
    Thus $g_{nm !}(\spSet_m)$ is $L$-dense in $\spSp_n$.
    Therefore $ f_{nm !}(\spSet_n) $ must be $2L$-dense in $\spSp_m$.
    
    On $\spSet_n \sqcup \spSet_m$ consider the symmetric function determined by the assignment
    \begin{equation}\label{eq:thm:GHdist_locLipEquiv_dist}
            \Dist_{nm}(x,y)
        \coloneqq
            \begin{cases*}
                \dist{x}{y}_n     &if $x,y \in \spSet_n$ \\
                \dist{x}{y}_m   &if $x,y \in \spSet_m$ \\
                    3L
                +   \inf\limits_{x' \in \spSet_n} \dist{x}{h_{nm}(x')} + \dist{f_{nm}(x')}{y}_n
                    &if $x \in \spSet_n, y \in \spSet_m$.
            \end{cases*}
    \end{equation}
    Obviously, $ \Dist_{nm} $ distinguishes points.
\begin{exercise}
    The triangle inequality is left to \cref{xca:GHdist_locLipEquiv_dist_triangle}.
\end{exercise}
\begin{exerciseSolved}
    Moreover it fulfills the triangle inequality: 
    for convenience set $ f\coloneqq f_{nm} $, $ g\coloneqq g_{nm} $, and $ h\coloneqq h_{nm} $.
    Take three points $x,y,z \in \spSet_n \sqcup \spSet_m$.
    
    The cases $x,y,z \in \spSet_n$ and $x,y,z \in \spSet_m$ are obvious.
    
    In case $x,y \in \spSet_n$ but $z\in\spSet_m$ observe that $ 
            \inf\limits_{x' \in \spSet_n} \dist{x}{h (x')} + \dist{f (x')}{z}_n
        \leq\inf\limits_{x' \in \spSet_n} \dist{x}{y} + \dist{y}{h (x')} + \dist{f (x')}{z}_n
    $ and, thus, $\Dist_{nm}(x,z) \leq \Dist_{nm}(x,y) + \Dist_{nm}(y,z) $.
    The case $x \in \spSet$ and $y,z\in\spSet_n$ is parallel.
    
    In the case  $x, z \in \spSet_n$ and $y\in\spSet_m$ observe 
    \begin{align*}
    \MoveEqLeft[2]
            \Dist_{nm}(x,y) + \Dist_{nm}(y,z) 
    \\
        &=   6L 
            + \inf\setBuilder
                {\dist{x}{h (y')} + \dist{f (x')}{y}_n + \dist{z}{h (z')} + \dist{f (z')}{y}_n }
                {x', z' \in \spSet}
    \\
        &\geq
            6L 
            + \inf\setBuilder
                {\dist{x}{h (x')} + \dist{f (x')}{f (z')}_n + \dist{z}{h (z')}}
                {x', z' \in \spSet} 
    \\
        &\geq
            5L 
            + \inf\setBuilder
                {\dist{x}{h (x')} + \dist{x'}{z'} + \dist{z}{h (z')}}
                {x', z' \in \spSet} 
    \\
        &\geq
            5L 
            + \inf\setBuilder
                {\dist{x}{h (x')} + \dist{h (x')}{h (z')} - \Norm{h }{diam}^{*\mathrm{L}}
            + \dist{z}{h (z')}}
                {x', z' \in \spSet} 
    \\
        &\geq
            \inf\setBuilder
                {\dist{x}{h (x')} + \dist{h (x')}{h (z')} + \dist{z}{h (z')}}
                {x', z' \in \spSet}
    \\
        &\geq
            \dist{x}{z}
    \text{.}
    \end{align*}
    
    In the remaining case $x, z \in \spSet_n$ and $y\in\spSet$ we get 
    \begin{align*}
    \MoveEqLeft[2]
            \Dist_{nm}(x,y) + \Dist_{nm}(y,z) 
    \\
        &=   6L 
            + \inf\setBuilder
                {\dist{y}{h (y')} + \dist{f (y')}{x}_n + \dist{y}{h (y'')} + \dist{f (y'')}{z}_n }
                {y', y'' \in \spSet} 
    \\
        &\geq
            6L 
            + \inf\setBuilder
                {\dist{h (y')}{h (y'')} + \dist{f (y')}{x}_n + \dist{f (y'')}{z}_n}
                {y', y'' \in \spSet} 
    \\
        &\geq
            5L 
            + \inf\setBuilder
                {\dist{f (y')}{f (y'')}_n + \dist{f (y')}{x}_n + \dist{f (y'')}{z}_n}
                {y', y'' \in \spSet} 
    \\
        &\geq    
            L + \dist{x}{z}
    \text{.}
    \end{align*}
 
\end{exerciseSolved}
    Within $(\spSet_n \sqcup \spSet_m, \Dist_{nm})$ we have $ \spSet_m^{5L]} \supseteq \spSet_n $
    since $ f_{nm}(\spSet_n) $ is $2L$-dense in $\spSp_m$.
    By the same fact, $ \spSet_n^{5L]} \supseteq ((f_{nm})_* M)^{2L]} \supseteq \spSet_m $.
    Hence $\Dist[GH](\spSp_n, \spSp_m) \leq 5L$.
\end{proof}

\begin{exercise}
\begin{xca}\label{xca:GHdist_locLipEquiv_dist_triangle}
    Check the triangle inequality for the symmetric function \cref{eq:thm:GHdist_locLipEquiv_dist}.
\end{xca}
\end{exercise}

The fact that the Gromov-Hausdorff space is complete \cite[11.1.1]{Petersen16} implies:
\begin{corollary}\label{cor:dilMetric}
    The space $ (\skeleton_0(\Cat{Met}, \Norm{}{diam}), \Dist[diam]^+ ) $ is a complete metric space.
\end{corollary}

\begin{exercise}
\begin{xca}\label{cor:MetCpt_normed}
    Show that the category $ (\Cat{Met}, \Norm{}{diam}) $ is normed. Hint: use \cref{lem:almostExpansiveEndo} for \cref{lst:def:norm1} and \cref{eq:snorm_dil_completion} for \cref{lst:def:norm2}.
\end{xca}
\end{exercise}

\begin{exerciseSolved}
\begin{corollary}\label{cor:MetCpt_normed}
    The category $ (\Cat{Met}, \Norm{}{diam}) $ is normed.
\end{corollary}
\begin{proof}
    For the first property, \cref{lst:def:norm1}, note that given two maps $f\colon \spSp \to \spSp*$ and $ g\colon \spSp* \to \spSp $ with $ \Norm{f}{diam} = \Norm{g}{diam} = 0 $ their compositions $ f\comp*g $ and $ g \comp* f $ have vanishing dilatation norm as well.
    \Cref{lem:almostExpansiveEndo} implies in this case that $ f\comp*g $ and $ g\comp*f $ have a dense image and are contractions.
    Hence both maps are also isometries in this set up. Thus by set theoretic arguments $f$ and $g$ are bijections.
    Since $f$ and $g$ are both expansions they also have to be contractions: assume for instance in the case of $f$ that the distance between two points $x,y\in \spSet$ is expanded, i.e.\ $ \dist x y < \dist {f(x)} {f(y)} $, then also $ \dist x y < \dist {g(f(x))} {g(f(y))} $ in contradiction to $\Norm{f\comp*g}{diam}^{*\mathrm{L}} = 0$.
    
    For the second property, \cref{lst:def:norm2}, take $ f_n\colon \spSp \to \spSp* $ with $\Norm{f_n}{diam} \to 0$.
    We do a diagonal argument.
    Take a dense sequence $x_1, \ldots \in \spSp$.
    For each $i=1,\ldots$ choose a sequence $x_{ij}$ such that $ x_{ij} \to x_i $ as $ j\to\infty $ and $ x_{ij} \in f_j^* \spSet* $ for each $j$.
    We choose a sequence $ j_1(1), j_1(2), \ldots $ such that for some $ y_{1n} = f_{j_1(n)} (x_{1n}) $ the sequence $y_n$ converges as $n\to\infty$ (using compactness of $\spSp*$).
    We proceed by choosing a further sequence $ j_2(1), j_2(2), \ldots $ such that some $ y_{2n} = f_{j_1(j_2(n))} (x_{2n}) $ converge as $n\to\infty$.
    We continue this procedure with sequences $ j_3(\blank), j_4(\blank), \ldots $.
    We define the densely defined function $f\colon \spSp\to\spSp*$ by $ 
        f(x_i) = \lim\limits_{n\to\infty} y_{jn} 
    $.
    By construction $\Norm{f}{diam} = 0$.
    Thus $\Norm{\bar f}{diam} = 0$ as follows from \cref{eq:snorm_diam_ptWise_sup} and an easy limit argument.
\end{proof}
\end{exerciseSolved}

\begin{workingNotes}

\section{Metric measure spaces}

\subsection{Prokhorov seminorm}
\label{ssec:Prokhorov}
\todo{the induced metric should be Lipschitz equivalent to the one from \cite{AbrahamDelmasHoscheit13}}
Let $\Cat{mMet}$ denote the category of metric measure spaces and Borel measurable maps.
Define the open and closed thickenings
\begin{align*}
    A^{r)} &\coloneqq \bigcup_{x \in A} \Openball{x}{r}
&&\text{for $r>0$ and}
\\
    A^{r]} &\coloneqq \bigcap_{r' > r} A^{r')}
&&\text{for $r\geq 0$.}
\end{align*}
Let $ \Cat{SmMet} $ be the category of metric measure spaces with scalars where
\begin{align*}
        \Cat{SmMet}_0
    &\coloneqq
        \setBuilder{(\spSp, r)}{ \spSp\in\Cat{mMet}_0 \text{ and } r \in [0,\infty] }
\\
        \hom[SmMet]{(\spSp, r)}{(\spSp*, s)}
    &\coloneqq
        \begin{cases*}
            \hom[mMet]{\spSp}{\spSp*}    & if $ r \leq s $,   \\
            \emptyset       & else.
        \end{cases*}
\end{align*}
Moreover let $ S\colon \Cat{mMet} \to \Cat{SmMet} $ denote the embedding that assigns to each metric measure space $\spSp$ the infinitely thickened space $(\spSp, \infty)$. This becomes a concrete category with distinguished subobjects by the assignment
\[
    \operatorname{SO}(\spSp) = \setBuilder{(\spSp[A], r)}{\spSp[A] \subseteq \spSp \text{ and } r \in (0,\infty) }
\]

The function
\begin{equation}
    c_P(A, v) \coloneqq \inf \setBuilder{\delta>0}{ \mu( A^{\delta)}) + \delta \geq v }
\end{equation}
is monotone as whenever $ (A, v) \subseteq (B, w) $, i.e.\ $A \subseteq B$ and $v \leq w$, we have for any $ \delta $ with $ \mu (A^{\delta)}) + \delta < v $ we have that $
    \mu (A^{\delta)}) + \delta < w
$. We call $c_P(A, v)$ the \definiendum{Prokhorov capacity}. Denote the \definiendum{Prokhorov seminorm} by $\Norm{}{P}$, defined using the Prokhorov capacity, i.e.\ \todo{probably reverse definition better, i.e.\ $\norm{}_{-c_P}$}
\begin{align*}
    \Norm{}{P} &\coloneqq \norm{}_{c_P}
\\
    &=
        \supBd_{\substack{v \geq 0,\\ A \in \mathcal{B}(\spSp*)} }
        \inf \setBuilder{\delta>0}{ \mu( (f^*A)^{\delta)}) + \delta \geq v }
        - \inf \setBuilder{\delta>0}{ \nu( A^{\delta)}) + \delta \geq v }
\text.
\end{align*}
We observe that for $\pmb{\emptyset} = (\emptyset, \Dist_{\pmb{\emptyset}}, 0)$ the initial object, i.e.\ empty space with zero measure, the volume of $\spSp$ can be expressed as the seminorm of the initial morphism
\begin{align*}
        \Norm{\pmb{\emptyset} \to \spSp }{P}
    &=  \supBd_v
        \inf\setBuilder{\delta>0}{ 0 + \delta \geq v }
        -\inf\setBuilder{\delta>0}{ \mu( M^{\delta)} ) + \delta \geq v }
\\
    &\leq\supBd_v v - (v - \mu(M))
\\
    &=\volume(\spSp)
\text.
\end{align*}

Define the Prokhorov distance from a measure $\spVol$ on a metric space to another such measure, $\spVol*$, by
\begin{align*}
        \Dist[P](\spVol, \spVol*)
    &\coloneqq
        \inf\setBuilder
            {\delta\in (0,\infty]}
            { \forall A\in \closed (\spSp) \colon \spVol( A^{\delta)} ) + \delta \geq \spVol*( A) }
\intertext{and the Prokhorov distance by symmetrization}
        \Dist[P]^+(\spVol, \spVol*)
    &\coloneqq
        \frac{1}{2} \left( \Dist[P](\spVol, \spVol*) + \Dist[P](\spVol*, \spVol) \right)
\text;
\end{align*}
to check that $\Dist[P]$ is a quasimetric observe that the separation $\Dist[P](\spVol, \spVol*) \neq 0 \implies \spVol \neq \spVol* $ is obvious, and that the remaining axioms (triangle inequality and identity of indiscernibles) are implied by the following \cref{prop:Prokhorov_dist_seminorm}.

\begin{proposition}\label{prop:Prokhorov_dist_seminorm}
    Given Borel measures $\spVol, \spVol*$ on a metric space $(\spSet, \spMet)$.
    Then 
    \[
        \Dist[P](\spVol, \spVol*) = \Norm{\id_{\spVol\spVol*}}{P}
    \]
    where  $\id_{\spVol\spVol*}$ is the identity map from $(\spSet, \spMet, \spVol)$ to $(\spSet, \spMet, \spVol*)$.
\end{proposition}

\begin{proof}
    By definition of the seminorm induced by a capacity \cref{eq:seminorm_by_w} we have
    \[
            \Norm{\id_{\spVol\spVol*}}{P}
        \coloneqq\supBd_{v, A}
                \inf \setBuilder{\delta>0}{ \nu( A^{\delta)}) + \delta \geq v }
            -   \inf \setBuilder{\delta>0}{ \mu( A^{\delta)}) + \delta \geq v }
    \text.
    \]

    First we show that $ \Dist[P](\spVol, \spVol*) \geq \Norm{\id_{\spVol\spVol*}}{P} $.
    Take any measurable $A \subseteq \spSet*$.
    Assume that $ \delta > \Dist[P](\spVol, \spVol*) $, i.e.\ $ \spVol( B^{\delta)} ) + \delta \geq \spVol*(B) $ for all measurable sets $B \subset \spSet$.
    We want to show $ \delta > \Norm{\id_{\spVol\spVol*}}{P} $.
    To this end we have to show that if $\spVol( A^{\delta_1)}) + \delta_1 < v$ for some $\delta_1 > \delta$, then
    $\spVol*( A^{\delta_2)}) + \delta_2 < v  $ for $\delta_2 = \delta_1 - \delta $.
    Indeed $
        v > \spVol( A^{\delta_1)}) + \delta_1 = \spVol( A^{\delta_2 + \delta)}) + \delta + \delta_1 
    $ and $ 
        \spVol( (A^{\delta_2)})^{\delta)} ) + \delta \geq \spVol*(A^{\delta_2)}) 
    $ imply $ \spVol*( A^{\delta_2)}) + \delta_2 < v $.
    
    The reverse inequality $ \Dist[P](\spVol, \spVol*) \leq \Norm{\id_{\spVol\spVol*}}{P} $ is shown by the parallel argument.
    Assume $ \delta < \Norm{\id_{\spVol\spVol*}}{P} $, i.e.\ $ \spVol( B^{\delta)} ) + \delta < \spVol*(B) $ for some measurable set $B \subset \spSet$.
    We want to show $ \delta < \Dist[P](\spVol, \spVol*) $, i.e.\   $ \inf \setBuilder{\delta>0}{ \nu( A^{\delta)}) + \delta \geq v }
            -   \inf \setBuilder{\delta>0}{ \mu( A^{\delta)}) + \delta \geq v } $ for any measurable set $A$ and $v\geq 0$.
    Choose $A = B$ and $ v = \spVol*(A) $.
    Then we have $ \inf \setBuilder{\delta>0}{ \nu( B^{\delta)}) + \delta \geq \spVol*(A) } > \delta $
    and $ \inf \setBuilder{\delta>0}{ \mu( A^{\delta)}) + \delta \geq \spVol*(A) } = 0 $.
    Thus $ \delta < \Norm{\id_{\spVol\spVol*}}{P} $.
\end{proof}

\begin{lemma}
    The Prokhorov quasidistance coincides with the Prokhorov distance for probability measures.
\end{lemma}

Observe that \todo{can be proved by thickening arithmetic from the appendix}
\begin{equation}\label{eq:ProkhorovDist_thickening}
    A^{\delta) \;\mathrm{C}\;\delta)} \subseteq A^{\mathrm{C}}
\text.
\end{equation}
This follows from the fact that $A^{\delta) \;\mathrm{C}\;\delta) \;\mathrm{C}} = A^{\delta) \;-\delta) }  \supset A^\circ $.

\begin{proof}
    To prove the lemma, we have to show that the quasimetric $\Dist[P]$ is symmetric.
    Assume that $ \Dist[P](\spVol, \spVol*) > \delta > 0 $.
    This must be witnessed by a measurable set $A\subset \spSet$ such that $
        \spVol( A^{\delta)} ) + \delta < \spVol*( A )
    $.
    Observe the equivalences
    \begin{align*}
        &&                      \spVol( A^{\delta)} ) + \delta     &< \spVol*( A )
    \\ &\iff&      1 - \spVol( A^{\delta)\;\mathrm{C}} ) + \delta    &< 1 - \spVol*( A^{\mathrm{C}} )
    \\ &\iff&                  \spVol*( A^{\mathrm{C}} ) + \delta  &< \spVol( A^{\delta) \;\mathrm{C}} )
    \\&\implies&                \spVol*( A^{\delta) \;\mathrm{C}\;\delta)} ) + \delta &< \spVol( A^{\delta) \;\mathrm{C}} )
    &&\text{by \cref{eq:ProkhorovDist_thickening}.}
    \end{align*}
    But this implies $ \Dist[P](\spVol*, \spVol) > \delta $.
    Since $\spVol, \spVol*$ were arbitrary, $ \Dist[P](\spVol*, \spVol) $ is symmetric.
\end{proof}

\begin{theorem}\label{thm:Prokhorov_dist}
    \begin{subequations}
    The Prokhorov norm can be expressed as
    \begin{align}
            \Norm{f}{P}
        &=
            \inf\setBuilder
                { \delta > 0 }
                {\forall A\in \mathcal{B}(\spSp*), r \geq 0\colon \spVol( (f^* A)^{r+\delta)} ) + \delta \geq \spVol*( (A)^{r)}) }
\\
\label{eq:thm:Prokhorov_dist_closed}
    &=
        \inf\setBuilder
            { \delta > 0 }
            {\forall A\in \mathcal{B}(\spSp*), r \geq 0\colon \spVol( (f^* A)^{r+\delta]} ) + \delta \geq \spVol*( A^{r]}) }
    \text.
    \end{align}
    The equations also hold with $\mathcal{B}(\spSp*)$ replaced by $\closed(\spSp*)$
    \end{subequations}
\end{theorem}

\begin{proof}
    Let $ \mathit{Ra}(f) $ and $ \mathit{Rb}(f) $ denote the right-hand side of the equations respectively. Let $ \mathit{Ra}_{\delta, A, r}(f) $ denote the fact $ \spVol (f^* A)^{r+\delta)}  + \delta \geq \spVol*A^{r)} $ and $ \mathit{Rb}_{\delta, A, r}(f) $ the fact $ \spVol (f^* A)^{r+\delta]}  + \delta \geq \spVol* A^{r)} $.
    
    First we prove $ \Norm{f}{P} \leq \mathit{Ra}(f), \mathit{Rb}(f) $.
    Take an arbitrary $\delta > \mathit{Ra}(f)$ (or, resp., $ \delta > \mathit{Rb}(f) $ and choose $\delta' \in (\mathit{Rb}(f), \delta)$).
    We want to prove $ \Norm{f}{P} < \delta $.
    This is to say that for every $v \geq 0$ and $ A \in \mathcal{B}(\spSp*)$ we have to show $ c_P(f^*A, v) - c_P(A, v) < \delta $.
    Observe
    \begin{align*}
         &       &   c_P(f^*A, v) - c_P(A, v) &< \delta
    \\
        &\iff   &   \inf \setBuilder{\delta>0}{ \spVol (f^*A)^{\delta)} + \delta \geq v } 
                        &< \delta + \inf \setBuilder{\delta>0}{ \spVol* A^{\delta)} + \delta \geq v }
    \\
        &\iff  &    \forall \delta_2\colon
                    \Bigl( \spVol* A^{\delta_2)} + \delta_2 \geq v 
                    &\implies \spVol (f^* A)^{\delta_2 + \delta)} + \delta_2 + \delta \geq v\Bigr)
    \end{align*}
    The implication follows from $\mathit{Ra}_{\delta_2, A, \delta}(f)$ (resp. from $\mathit{Ra}_{\delta_2, A, \delta'}(f)$ by $ 
        v   \leq \spVol* A^{\delta_2)} + \delta_2
            \leq \spVol* A^{\delta_2]} + \delta_2
            \leq \spVol (f^* A)^{\delta' + \delta_2]} + \delta' + \delta_2
            \leq \spVol (f^* A)^{\delta + \delta_2)} + \delta + \delta_2
    $).
    
    In the opposite direction we take any $l$ with $ \Norm{f}{P} < l $.
    This is to say that for all $v \geq 0$ and $ A \in \mathcal{B}(\spSp*)$ we have $c_P(f^*A, v) - c_P(A, v) < l$.
    We have to prove $ \mathit{Ra}(f), \mathit{Rb}(f) \leq l $ that is to say that for any $r \geq 0$ and $ B \in \mathcal{B}(\spSp*)$ we have $ \mathit{Ra}_{l, B, r}(f) $ and, resp., $\mathit{Rb}_{l, B, r}(f)$.
    Choose $A = B$ and $v = \spVol* B^{r)} + r $ (resp. $v = \spVol* B^{r]} + r $).
    Thus $c_P(A, r) = r$ (for both choices of $v$).
    Hence $ c_P(f^*A, v) - c_P(A, v) < l$ implies
    \[
         \inf \setBuilder{\delta>0}{ \spVol( (f^*A)^{\delta)}) + \delta \geq \spVol* B^{r)} + r } < l + r
    \text,
    \]
    thus $ \spVol (f^*A)^{l+r)} + (l+r) \geq \spVol* B^{r)} + r $ but this is $\mathit{Ra}(f)$;
    resp.\ we have 
    \[
         \inf \setBuilder{\delta>0}{ \spVol( (f^*A)^{\delta)}) + \delta \geq \spVol* B^{r]} + r } < l + r
    \text,
    \]
    thus $ \spVol (f^*A)^{l+r]} + (l+r) \geq \spVol* B^{r]} + r $ but this is $\mathit{Rb}(f)$.
    
    The replacement of $\mathcal{B}(\spSp*)$ by $\closed(\spSp*)$ is an immediate consequence of the fact that each set $ A^{r+\delta)} $ or $ A^{r+\delta]} $ contains the closure of $A$.
\end{proof}

Let $\Cat{mMet}_{\textnormal{fin}}$ denote the category of fully supported finite mm spaces.

\begin{theorem}
    On the category $\Cat{mMet}_{\textnormal{fin}}$ the seminorm $ \Norm{}{P} $ is a norm.
    \todo{Need probably more conditions}
\end{theorem}

\begin{proof}
    We will first proof axiom \cref{lst:def:norm1}.
    
    First note that $\spVol(\spSet) = \spVol*(\spSet*) $ because by \cref{eq:thm:Prokhorov_dist_closed} for all $\delta > 0$ we have $ \spVol(\spSet) + \delta \geq \spVol*(\spSet*) $ and $ \spVol*(\spSet*) + \delta \geq \spVol(\spSet) $. 
    Without loss of generality assume $\spVol(\spSet) = \spVol*(\spSet*) = 1$.
    %Moreover we have
    %\[
    %    \spVol( (f^* A)^{r+\delta]} ) + \delta = \spVol*( A^{r]})  
    %\]
    %for all $r\geq 0$ and $A\in \mathcal{B}(\spSp*)$ because 
    taking any closed set

    Assuming  $\spVol( (f^* A)^{r+\delta]} ) + \delta < \spVol*( A^{r]})  $ we have
    \begin{align*}
        &&    \spVol( (f^* A)^{r+\delta]} ) + \delta &< \spVol*( A^{r]})
    \\ &\iff&      1 - \spVol( (f^* A)^{r + \delta]\;\mathrm{C}} ) + \delta    &< 1 - \spVol*( A^{r]\mathrm{C}} )
    \\ &\iff&                  \spVol*( A^{r]\mathrm{C}} ) + \delta  &< \spVol( (f^* A)^{\delta] \;\mathrm{C}} )
    \\&\implies&                \spVol*( A^{r + \delta] \;\mathrm{C}\;\delta]} ) + \delta &< \spVol( (f^* A)^{\delta] \;\mathrm{C}} )
    &&\text{by \cref{eq:ProkhorovDist_thickening}.}
    \end{align*}
    
\end{proof}

\subsection{Wasserstein seminorm}
\label{ssec:snorm_Wasserstein}
We denote the for $\varphi\colon \spSet \to \R$ by
\[
    \hoeld{\varphi}1 \coloneqq \Norm{\varphi}{Lip}^{*\mathrm{L}} = \sup_{x, y \in \spSet} \frac{\abs{\varphi(x) - \varphi(y)} }{\dist{x}{y}}
\]
what is normally called the Lipschitz seminorm in function space theory.
The category $\Cat{PmMet}$ of projective metric measure spaces has as object equivalence classes $\equivClass\spSp {W}$ of mm-spaces $\spMetVolDef$ with $\spVol \neq 0$ under the equivalence relation $ 
  (\spSet, \lambda\spMet, \spVol ) \equivRel{W} (\spSet, \spMet, \lambda \spVol )
$ for any $\lambda>0$ and as morphisms any measurable maps.
Note that $ 
        \hoeld{(\spSet, \nicefrac{1}{\lambda}\spMet, \lambda\spVol) \xrightarrow{f} \R }1  
    =   \lambda
        \hoeld{(\spSet, \spMet, \spVol) \xrightarrow{f} \R }1
$.

We define the space of \definiendum{projective metric measure spaces with test functions $\Cat{TPmMet}$} as 
\begin{align*}
        \Cat{TPmMet}_0
    &\coloneqq
        \setBuilder*
            { (\equivClass\spSp{W}, \varphi) }
            { \begin{gathered}
                \spSp \in \Cat{PmMet}_0, \\
                \varphi\colon \spSet \to [0,1] \text{ is measurable}
            \end{gathered}}
\\
        \hom[TPmMet]{(\spSp, \varphi)}{(\spSp*, \psi)}
    &\coloneqq
        \setBuilder*
            { f }
            { \begin{gathered} 
                f \in \hom[PmMet]{\spSp}{\spSp*} \text{ with} \\
                \varphi \leq \psi \comp f \text{ and }\\ 
                \hoeld{\varphi}1 \geq \hoeld{\psi \comp f}1
            \end{gathered}}
\end{align*}
where the condition $\hoeld{\varphi}1 \geq \hoeld{\psi \comp f}1$ is independent of the choice of representative of $\equivClass\spSp {W}$.
We embed $\Cat{PmMet}$ into $\Cat{TPmMet}$ via the functor 
\begin{equation}
    \Cat{PmMet} \to \Cat{TPmMet}, \quad
    \spSp \mapsto ( \spSp, 1 )
\end{equation}
where $1$ denotes the constant function with value $1$.

The set of subobjects, $\subObjs((\equivClass{(\spSet, \spMet, \spVol)}{P}, \varphi))$, consists of those ordered pairs $(\equivClass{\spSp[A]}{P}, \varphi')$ with $\spSp[A] = (A, \spMet[A], \spVol[A]) \in \Cat{mMet}_0 $ such that  $A \subseteq \spSet$, the inclusion function of $A$ in $\spSet$ is $(\mathcal{B}(\spSp[A]), \mathcal{B}(\spSp))$-measurable, $\varphi' \leq \varphi$, and $\hoeld{\varphi'}1 \geq \hoeld{\varphi}{1}$ as functions on  $(A, \spMet[A])$.
We define the precapacity, $c_{\mathrm{W}}( [\spSp], \varphi )$, of $( [\spSp], \varphi )$ by 
\[
        \capacity[W]( [\spSp], \varphi )
    \coloneqq
        \supBd \setBuilder*
            { \int \varphi d \spVol }
            { (\spSet, \spMet, \spVol) \in \equivClass\spSp{W}, 
                \hoeld*{(\spSet, \spMet) \xrightarrow{\varphi} [0,1] }1 \geq 1 }
\text.
\]

If $\varphi\colon \spSp \to [-1,1]$ is not constant (i.e.\ $\hoeld{\varphi}1 > 0$) and $\hoeld{\varphi}1 < \infty$ define $(\spSet, {\spMet}_\varphi, \spVol_\varphi)$ to be the representative of $\spSp$ such that $ \hoeld{\varphi\colon (\spSet, \spMet) \to \Bar{\R}}1 = 1$; i.e.\ 
\[
    {\spMet}_\varphi \coloneqq \hoeld{\varphi}1 \spMet, \qquad
    \spVol_\varphi \coloneqq \frac{1}{\hoeld{\varphi}1} \spVol
\]

\begin{lemma}
    \begin{subequations}
    We have
    \begin{empheq}[left={c_{\mathrm{W}}( \spSp, \varphi )
    =\empheqbiglbrace~}]{align}
    \label{eq:capacityWasserstein_main}
    &\smallint \varphi d \spVol_\varphi
        && \text{when }\hoeld{\varphi}1 \in (0, \infty),
    \\
    \label{eq:capacityWasserstein_infty}
    &\infty  && \text{if } \hoeld{\varphi}1 = 0 \text{ and } \varphi > 0,
    \\
    \label{eq:capacityWasserstein_zero}
    &0   && \text{if } \hoeld{\varphi}1 = \infty \text{ or } \varphi \equiv 0.
    \end{empheq}
    \end{subequations}
\end{lemma}

The cases \cref{eq:capacityWasserstein_infty,eq:capacityWasserstein_zero} are called the \definiendum{trivial cases}.

\begin{proof}
 \todo{fill in}    
\end{proof}

\begin{lemma}
    The precapacity $ \capacity[W] $ is a capacity.
\end{lemma}

\begin{proof}
    To check monotonicity take a commuting triangle 
\begin{center}
\begin{tikzcd}
        & (\equivClass\spSp{P}, 1) &           \\
    (\equivClass{\spSet, \spMet, \spVol}{P}, \varphi ) \ar[ru, "{\id}"] \ar[rr, "{\id}"] & &  (\equivClass{\spSet, \spMet', \spVol'}{P}, \psi ) \ar[lu, "{\id}"] 
\end{tikzcd}
\end{center}
where $ (\spSet, \spMet, \spVol)$ and $ (\spSet, \spMet', \spVol') $ are in $ \equivClass\spSp{P} $, $ \hoeld{\psi}{1} \leq \hoeld{\varphi}{1}$. We have to show $ \capacity[W]( \equivClass\spSp{P}, \varphi ) \leq \capacity[W]( \equivClass\spSp{P}, \psi ) $.

We first consider the trivial cases.
The case $ \hoeld{\psi}{1} \in \{ 0, \infty \} $ is obvious: either $ \hoeld{\psi}{1} = 0 \vee \psi > 0 $, $\hoeld{\psi}{1} = \infty$, or $\varphi \equiv 0$.
In the first case we have $ \capacity[W]( \equivClass\spSp{P}, \psi ) = \infty \geq \capacity[W]( \equivClass\spSp{P}, \varphi ) $.
In the second case we have $ \hoeld{\varphi}{1} \geq \hoeld{\psi}{1} = \infty $ and, hence, $ \capacity[W]( \equivClass\spSp{P}, \varphi ) = \infty $.
In the last case we have $\varphi \leq \psi \equiv 0$ and thus $ \capacity[W]( \equivClass\spSp{P}, \varphi ) = 0 $.
In the case $ \hoeld{\varphi}{1} \in \{ 0, \infty \} $ we have either $ \capacity[W]( \equivClass\spSp{P}, \varphi ) = 0 $, and thus the inequality $ \capacity[W]( \equivClass\spSp{P}, \varphi ) \leq \capacity[W]( \equivClass\spSp{P}, \psi ) $ hold automatically, or we have $ \hoeld{\varphi}{1} = 0 $ and $ \varphi > 0$ but the we have also $ \hoeld{\psi}{1} = 0 $ and $ \psi > 0$. Thus also $ \capacity[W]( \equivClass\spSp{P}, \psi ) = \infty $.

and $ (\spSet, \spMet, \spVol) $, $ (\spSet, \spMet', \spVol') $ are chosen extremely, i.e.\  $ \hoeld{(\spSet, \spMet) \xrightarrow{\varphi} [0,1] }1 = \hoeld{(\spSet, \spMet') \xrightarrow{\psi} [0,1] }1 = 1$.

We know: $ \spVol = \lambda \spVol' $ and $ \lambda \spMet = \spMet'  $

$ \hoeld{\psi}{1} \leq \hoeld{\varphi}{1}$ implies that $ \hoeld{(\spSet, \spMet) \xrightarrow{\psi} [0,1] }1 \leq 1
 $. Moreover $ \hoeld{(\spSet, \lambda\spMet) \xrightarrow{\psi} [0,1] }1 = 1
 $.
Hence $ \lambda \leq 1  $.
Thus
\[
            \capacity[W]( \equivClass\spSp{P}, \varphi )
    =       \int \varphi d \spVol
    \leq    \int \psi d \spVol
    \leq    \int \psi d \frac{\spVol}{\lambda}
    =       \int \psi d \spVol'
    =       \capacity[W]( \equivClass\spSp{P}, \psi )
\text.
\qedhere
\]
\end{proof}

\begin{subequations}
\begin{align}
    \Norm{f}{W} &\coloneqq \norm{f}_{-c_W}
\\ \nonumber
    &=   \supBd c_{\mathrm{W}}( \spSp, f^*\psi ) - c_{\mathrm{W}}( \spSp*, \psi )
\\
    &=   \supBd c_{\mathrm{W}}( \spSp, \psi \comp f ) - c_{\mathrm{W}}( \spSp*, \psi )
\end{align}
\end{subequations}

For general $ \mathrm{W}^p $ spaces generalize the construction by introducing the equivalence relation
\begin{equation}
    (\spSet, \lambda\spMet, \spVol) \sim_p (\spSet, \spMet, \lambda^p\spVol)
\end{equation}
for $\lambda > 0$. 
Let $\Cat{P${}^p$mMet}$ be the category of such spaces and measurable maps.
Let $\mathit{Mar}^p_f(\spVol, \spVol*)$ denote the set of all measures $m$ on $\spSet \times \spSet*$ such that $(\proj_{\spSet*})_* m = s\spVol*$ for some $s > 0$, $(\proj_{\spSet})_* m = r^p\spVol$ for some $r > 0$, and $ \hoeld{f\colon (\spSet, \nicefrac{1}{r} \spMet ) \to (\spSet*, \spMet*)}1 \leq 1 $

Observe that
\begin{equation}\label{eq:Wasserstein_totalMeasures}
    \mathit{Mar}^p_f(\mu, \lambda^{-p}\nu) = \lambda^{-p} \mathit{Mar}^p_f(\mu, \nu) 
\end{equation}
since for every $m \in \mathit{Mar}^p_f(\mu, \nu)$ we have $(\proj_{\spSet*})_* (\lambda^{-p} m) = \lambda^{-p} s\spVol*$ for some $s > 0$, 
$(\proj_{\spSet})_* ( \lambda^{-p} m ) = (\lambda^{-1} r)^p\spVol$ for some $r > 0$, and $ \hoeld{f\colon (\spSet, \nicefrac{\lambda}{r} \spMet ) \to (\spSet*, \lambda \spMet*) }1 \leq 1 $; and vice versa $ \lambda^p m \in \mathit{Mar}^p_f(\mu, \nu)$ for every $ m \in \mathit{Mar}^p_f(\mu, \lambda^{-p}\nu) $.
Define the quantity
\begin{equation}
    \Norm{f}{W${}^p$}
    \coloneqq
        \inf_{m \in \mathit{Mar}^p_f(\mu, \nu)}
        \sqrt[\leftroot{3}\uproot{3} \scriptstyle p]{\int \dist{f(x)}{y}^p d m }
\end{equation}
Note that by \cref{eq:Wasserstein_totalMeasures} this quantity does not depend on the representative of $\spSp*$.
The independence from the choice of the representative of $ \spSet $ is immediate from the definition of $\mathit{Mar}^p_f(\mu, \nu)$.

\color{red}

$ |m| = |\nu| $

\begin{example}
    Look at the case that $ f $ is constant with value $q$.
    Then
    \begin{equation}
    \Norm{f}{W${}^p$}
    =
        \inf_{m \in \mathit{Mar}^p_f(\mu, \nu)}
        \sqrt[\leftroot{3}\uproot{3} \scriptstyle p]{\int \dist{q}{y}^p d \nu(y) }
\end{equation}

For $p=1$ this should be equal to
\[
    \norm{f}_{c_W}
    =   \supBd_{\psi\leq 1} c_{\mathrm{W}}( \spSp*, \psi ) - c_{\mathrm{W}}( \spSp, \psi (q) )
\]
i.e.\
\[
    \supBd_\psi (-c_{\mathrm{W}}( \spSp*, \psi )) - (-c_{\mathrm{W}}( \spSp, \psi (q) ))
    =
    \int \dist{q}{y} d \nu(y)
\]
idea $\psi(x) = \dist{q}{x} \wedge 1$

\end{example}

\color{black}

\begin{theorem}[Kantorovich-Rubinstein]
    If $\spSp$, $\spSp*$ have finite volume,
    \[
            \Norm{}{W} 
        = 
            \Norm{f}{W${}^1$}
          + \left( 1 - \hoeld{ 
                (\spSet, \left| \spVol \right| \spMet ) \xrightarrow{f} 
                (\spSet*, \left| \spVol* \right| \spMet* ) 
            }1 \right)
    \]
\end{theorem}

\begin{proof}
    Easy direction: observe that for $\mu$ in definition of $\Norm{}{W}$ the variable $r$ from definition of $\mathit{Mar}^1_f(\mu, \nu)$ is 1.
    \begin{align*}
            \int \psi d (f_* \spVol - \spVol*)
        &=     \int \psi \circ f d \spVol
            - \int \psi d \spVol* 
    \\
        &=     \int \psi(x) \circ f d m(x,y)
            - \int \psi(y) \cdot \frac{\volume \spVol*}{\volume \spVol} d m (x, y)
    \end{align*}
\end{proof}

\end{workingNotes}

\appendix

\section{Notation}\label{sec:notation}

In this appendix we will describe much notation used throughout this work.  Other notation can be found in our references or is defined along the way in the body of this article.

\subsection{Set theory}\label{ssec:setTh}
Per common practice, we will typically use the notation of  Zer\-me\-lo-Fraen\-kel Set Theory with the Axiom of Choice (ZFC), but in many cases, especially when we need to work with proper classes, we will actually use the relatively consistent extension referred to generally as GBN (Gödel-Bernaise-von Neumann) Class-Set Theory with the Axiom of Choice.  Foundationally, we could use instead category theoretic foundations, but that seems to us to be merely a matter of taste.  Accordingly, we leave it to the readers to adjust our recipes and seasonings for the dishes we describe to their preferences. 

Given a set $X$, we denote by $\powerSet (\set)$ the power set of $X$, i.e. the set of all subsets of $\set$.  We will find it convenient to have a systematic notation for image and preimage mappings on the power set of a set, but in fact, for more variants than merely the simplest pair of such, whence, we denote by $f_*(A)$, or just $f_*A$ the set $\setBuilder{f(x)}{ x\in A }$, where $A$ is a subset of the domain of $f$, denote the preimage of any set $B$ under $f$ by $f^*(B)$ or $f^*B$ and by $f_!A$ or $f_!(A)$ we denote the \textbf{small image} of $A$, $f_!(A) \coloneqq \setBuilder{y}{ f(x)=y \implies x \in A }$ (note that $f_! (A)$ always contains all points in $\set*$ that are not hit by $f$).
One should be careful that the adjunctions for $ \powerSet (\set), \powerSet(\set*)$ are reversed compared to the adjunctions for a geometric morphism in algebraic geometry. This is because in algebraic geometry one would not study $\powerSet(\set)$ or $\powerSet(\set*)$ but sheaves thereon. Thus our adjunctions are
\[
  f_* \dashv f^* \dashv f^!.
\]
Note that both $f_!$ or $f_!$ are completely determined by this property, as---dwelling in the set-up of posets---these adjunctions are actually Galois connections.
%Also, by abuse of notation, we do not distinguish between the version of $f^*$ that is a ''function'' (having a graph which actually is a set) defined only on the power set of the codomain of $f$ and the version defined on the class of all sets, which is actually a ''Function'', whose graph is, formally, itself a proper class.

\subsection{Orders}
We denote partial orders by $ \ordDef, \ordDef*, $ and $ \ordDef** $.
Let $\mathcal{L} = (L, \leq)$ be a complete lattice, i.e.\ a poset that admits all suprema (and consequently, all infima).  We define $\sup\colon L\times \powerSet (L) \to L$ by 
\begin{align*}
    \sup (m, M) &\coloneqq \sup \left( M \cap \setBuilder{m'}{m\leq m'}  \right) \cup \{m\},
\intertext{and then for any $m\in L$, set }
    \supBd[m] M &\coloneqq  \sup (m, M)
\intertext{for any subset $M$ of $L$. Moreover, if $I$ is any (indexing) set, then}
    \smash{\supBd[m]_P f} &\coloneqq \supBd[m] \setBuilder{f(i)}{i\in P},
\end{align*}
for any subset $P$ of $I$ and function $f$ on $I$. If not specified otherwise, $\sup$ is understood as the supremum function on the extended real numbers $[-\infty,\infty]$.

\subsection{Categories}

We list the (2-)-categorical notation used in this article:
\begin{itemize}
    \item special (bi-)categories: $\Cat{Top}, \Cat{Met}$, etc.;
    \item variables for (bi-)categories: $ \cat, \cat*, \cat** $;
    \item variables for objects: $X, Y, M, \ldots$;
    \item variables for morphisms $f$, $g$, $h$;
    \item $f\colon \source(f) \to \target(f)$;
    \item variables for 2-morphisms: $\alpha\colon f \2to g$, $ \beta\colon g \2to h$, etc.;
    \item $ \cat_0 $/$ \cat_1 $/$ \cat_2 $ for the collection of objects/morphisms/2-morphisms
    \item $ \hom{X}{Y} $ set of morphisms from $X$ to $Y$ in category $\cat$, written $\hom*{X}{Y}$ if category is specified by the context;
    \item composition $f \comp* g = g \circ f$ of morphisms;
    \item vertical composition $ \alpha \compV* \beta = \beta \compV \alpha $ of 2-morphisms;
    \item horizontal composition $ \alpha \compH* \beta = \beta \compH \alpha $ of 2-morphisms;
    \item 
        we write $ {\comp*_{\cat}} = {\comp*}$, ${\comp_{\cat}} = {\comp}$, 
        ${\compV*_{\cat}} = {\compV*}$, ${\compV_{\cat}} = {\compV}$, 
        ${\compH*_{\cat}} = {\compH*}$, ${\compH_{\cat}} = {\compH} $ 
        to specify the category.
\end{itemize}

\section{2-categorical viewpoint on seminorms}
    \label{sec:2cat_viewPoint}
Let $\Cat{Cat}$ denote the category of small categories. Note that this category has products.
Recall that a \definiendum{bifunctor from $\cat$ and $\cat*$ to $\cat**$ } is a functor from a product category $\cat \times \cat*$ to $\cat**$.

\subsection{Strict 2-categories}
A \definiendum{strict 2-category} $\cat$ is a category enriched in $\Cat{Cat}$ meaning that $\cat$ consists of
\begin{itemize}
    \item a class $ \cat_0 $ of objects;
    \item for each $X, Y \in \cat_0 $ a category $ \hom{X}{Y} \in \Cat{Cat} $.
        Morphisms in $ \hom{X}{Y} $ are called \definiendum{2-morphisms} and depicted by ${\Rightarrow}$. Composition of such morphisms is called \definiendum{vertical composition} and denoted by ${\compV*}$.
    \item
        For all objects $ X, Y, Z \in \cat_0 $ there is a bifunctor $ 
            \compH*_{XYZ} = \compH* \colon  \hom{X}{Y} \times \hom{Y}{Z} \to \hom{X}{Z} 
        $, $ (f,g) \mapsto f \compH* g \eqqcolon g \compH f $, $ (f,g) \mapsto \alpha \compH* \beta \eqqcolon \beta \compH \alpha $, called \definiendum{horizontal composition}.
    \item 
        for each object $X$ an \definiendum{identity element} $\id_X \in \hom{X}{X}$.
\end{itemize}
such that the following axioms are satisfied
\begin{itemize}
    \item  
        for object $X, Y, Z, X'$ an \textit{associativity law}, the equality of functors
        \[
                \compH*_{XYX'} \comp_{\Cat{Cat}} ( \id_{\hom{X}{Y}} \times ( \compH*_{YZX'} ) )
            =
                \compH*_{XZX'} \comp_{\Cat{Cat}} ( ( \compH*_{XYZ}) \times \id_{\hom{Z}{X'}} )
        \]
        (which for morphisms $ X \xrightarrow{f} Y \xrightarrow{g} X \xrightarrow{h} X' $ states that $ f \compH* (g \compH* h) = ( f \compH* g ) \compH* h $).
    \item  
        for each $X \xrightarrow{f} Y \in \cat_1 $ the \textit{identity law}
        \[
            f = \id_X \compH* f = f \compH* \id_Y
        \text{.}
        \]
\end{itemize}
Note that strict 2-categories are too restrictive for many applications. Therefore there are weaker notions of a 2-category, especially bicategories.

\subsection{Lax functors}
In the case of strict 2-categories a \definiendum{lax functor} $F\colon \cat \to \cat*$ can be defined as an assignment of
\begin{enumerate}[label=($l$F\arabic*)]
    \item
        each $ X \in \cat_0 $ to an object $ F_X \in \cat*_0 $;
    \item
        each $ \hom{X}{Y} $ to a functor $ F_{XY} \colon \hom{X}{Y} \to \hom{F_X}{F_Y} $;
    \item (lax preservation of identity)
        each $ X \in \cat_0 $ to  an invertible 2-morphism %, $F_{\id_X}$, such that 
        $ F_{\id_X}\colon \id_{F_X} \Rightarrow F_{XX}(\id_X) $  in $ \cat*_2 $;
        % comparison to strict 2-functor: \id_{F_X} => F_{XX}(\id_X)
    \item (lax preservation of composition)
        each $ X, Y, Z \in \cat_0 $ to a natural transformation $F_{XYZ}$ from the bifunctor $ (f, g) \mapsto P_{XY}(f) \compH* P_{YZ}(g) $, $  \hom{X}{Y} \times \hom{Y}{Z} \to \hom{P_X}{P_Z} $ to the bifunctor $ (f,g) \mapsto P_{XZ}(f \compH* g) $;
\end{enumerate}
such that
\begin{enumerate}[label=($l$F\arabic*),resume]
    \item
        for each $ X, Y \in \cat_0 $ and $ f \in \hom{X}{Y} $
        an \textit{identity law}
        \begin{align}
        \tag{$l$F5a}
                (F_{\id_X} \compH* \id_{F_{XY}(f)}  ) \compV* F_{XXY}( \id_X, f ) 
            = 
                \id_{F_{XY}(f)}
        \\
        \tag{$l$F5b}
                ( \id_{F_{XY}(f)} \compH* F_{\id_Y} ) \compV* F_{XYY}( f, \id_Y ) 
            = 
                \id_{F_{XY}(f)}
        \end{align}
    \item
        for each diagram $ X \xrightarrow{f} Y \xrightarrow{g} X \xrightarrow{h} X' $
        an \textit{associativity law}
        \begin{multline*}
              ( F_{XYZ}(f, g) \compH* \id_{F_{ZX'}(h)} ) \compV* F_{XZX'} ( f \comp* g, h )
            \\=
              ( \id_{F_{XY}(f)} \compH* F_{YZX'}(g, h) ) \compV* F_{XYX'} (f, g \comp* h)
        \end{multline*}
\end{enumerate}
Note that for more general bicategories the last two properties become more complicated.
If the units $ F_{\id_X}$ are identities, i.e.\ $ F_{\id_X} = \id_{\id_{F_X} } $, the lax functor is called \definiendum{normal}.

\subsection{Seminorms as lax functors}
        \label{ssec:seminorms_laxFunctors}

Remember that every set $S$ can be regarded as a category by interpreting $ S $ as the set of objects and allowing only trivial morphisms, $ \hom{x}{y} = \emptyset $ for $x\neq y$ and $ \hom{x}{x} = \{ \id_{x} \} $.
\begin{subequations}
In the same manner any (ordinary) category $\cat$ can be regarded as a strict 2-category by defining $\hom{X}{Y}_0 = \hom{X}{Y} $ and $ \hom{f}{g} = \emptyset $ for distinct $f,g\in \hom{X}{Y}$ or $\hom{f}{f} = \{\id_f\}$.
Another example is given by a $\Cat*{( *, [0,\infty], +, \geq)}$ where
\begin{align}
    \Cat*{(*, [0,\infty], +, \geq)}_0 &\coloneqq \{*\}
\\
    \hom{*}{*}[\Cat*{(*, [0,\infty], +, \geq)}] &\coloneqq \Cat*{([0,\infty], \geq)}
    &&\text{with } \compV*  = {\geq}
\\
\label{eq:R_comp}
    r \compH* s &\coloneqq r + s
\text{.}
\end{align}
One immediately checks that \cref{eq:R_comp} is functorial from the fact that $r \geq r'$ and $s \geq s'$ implies $ r + s \geq r' + s' $.
\end{subequations}

\begin{proposition}
    \begin{subequations}
    A seminorm $ \norm{}\colon \cat \to [0,\infty] $ on a (1-)category $\cat$ forms a lax functor $ \norm{}_{\mathrm{lF}}\colon \cat \to \Cat*{( *, [0, \infty], +, \geq)} $ by the assignments
    \begin{align}
        F_X &\coloneqq  *         
    &&\text{for all } X \in \cat_0
    \\
        F_{00}(f) &\coloneqq  \norm{f}       
    &&\text{for all } f \in \cat_1
    \\
    \label{eq:sN_2morph}
        F_{00}(f) (\id_f) &\coloneqq  (\norm{f} \geq \norm{f})       
    &&\text{for all } f \in \cat_1
    \\
        F_{\id_X} &\coloneqq (0 \geq 0)
    &&\text{for all } X \in \cat_0
    \\
    \label{eq:sN_comp}
            F_{000}(f,g) 
        &\coloneqq 
            ( \norm{f} + \norm{g} \geq \norm{f \comp* g}  )
    &&\text{for all } X \xrightarrow{f} Y \xrightarrow{g} Z
    \text{.}
    \end{align}
    \end{subequations}
\end{proposition}

\begin{proof}
    First observe that $F$ is well-defined: the function defined by \cref{eq:sN_2morph} is in fact total since the only 2-morphisms in $\cat$ are identity morphisms.
    In \cref{eq:sN_comp} the image exists by triangle inequality.
    
    The identity laws hold automatically since the 2-homsets $ 
        \hom{r}{s}
            [{\hom{*}{*}[\Cat*{( *, [0, \infty], +, \geq)}]}]
    $ are at most singletons for all $r, s \in [0, \infty]$. The same applies to the associativity law.
\end{proof}

Note that any lax functor to $\Cat*{( *, [0, \infty], +, \geq)}$ is automatically normal.

\printbibliography

\end{document}